\documentclass[11pt]{amsart}
\usepackage[utf8]{inputenc}

\usepackage[enableskew]{youngtab}
\usepackage[usenames,dvipsnames]{xcolor}

\usepackage{subcaption, pgf, tikz,geometry, hyperref}
\usepackage{nicematrix}

\usetikzlibrary{arrows, automata, calc, positioning}

\tikzset{every state/.style={minimum size=0pt}}

\usepackage[titletoc,title]{appendix}

\usepackage{amsmath,amsfonts,amssymb,mathtools,stmaryrd, faktor, mathrsfs,amsthm, multicol, comment, mathrsfs}
\usepackage{dsfont}

\theoremstyle{definition}
\newtheorem{theorem}{Theorem}[section]
\newtheorem{thmx}{Theorem}[section]

\newtheorem{proposition}[theorem]{Proposition}
\newtheorem{lemma}[theorem]{Lemma}
\newtheorem{corollary}[theorem]{Corollary}
\newtheorem{definition}[theorem]{Definition}
\newtheorem{example}[theorem]{Example}
\newtheorem{remark}[theorem]{Remark}
\newtheorem{conjecture}{Conjecture}

\newcommand{\N}{\mathbb{N}}
\newcommand{\F}{\mathbb{F}}
\newcommand{\K}{\mathbb{K}}
\newcommand{\Z}{\mathbb{Z}}
\newcommand{\Q}{\mathbb{Q}}
\newcommand{\R}{\mathbb{R}}

\newcommand{\T}{\mathcal{T}}
\renewcommand{\O}{\mathfrak{O}}
\renewcommand{\P}{\mathfrak{P}}
\newcommand{\h}{\mathfrak{h}}

\newcommand{\GL}{\mathrm{GL}}

\newcommand{\stab}{\text{Stab}}

\newcommand{\sym}{\mathcal{S}}

\expandafter\def\expandafter\normalsize\expandafter{%
    \normalsize%
    \setlength\abovedisplayskip{3pt}%
    \setlength\belowdisplayskip{3pt}%
    \setlength\abovedisplayshortskip{-4pt}%
    \setlength\belowdisplayshortskip{2pt}%
}

\usepackage{graphicx,float}
\usepackage{tikz-cd}

\usepackage[ruled,vlined]{algorithm2e}
\usepackage{algorithmic}

\usepackage{minted}
\makeatletter
\@namedef{subjclassname@2020}{\textup{2020} Mathematics Subject Classification}
\makeatother

\makeatletter
\def\@tocline#1#2#3#4#5#6#7{\relax
  \ifnum #1>\c@tocdepth 
  \else
    \par \addpenalty\@secpenalty\addvspace{#2}%
    \begingroup \hyphenpenalty\@M
    \@ifempty{#4}{%
      \@tempdima\csname r@tocindent\number#1\endcsname\relax
    }{%
      \@tempdima#4\relax
    }%
    \parindent\z@ \leftskip#3\relax \advance\leftskip\@tempdima\relax
    \rightskip\@pnumwidth plus4em \parfillskip-\@pnumwidth
    #5\leavevmode\hskip-\@tempdima
      \ifcase #1
       \or\or \hskip 1em \or \hskip 2em \else \hskip 3em \fi%
      #6\nobreak\relax
    \hfill\hbox to\@pnumwidth{\@tocpagenum{#7}}\par
    \nobreak
    \endgroup
  \fi}
\makeatother

\usepackage[
    style=numeric,
    giveninits=true
]{biblatex}

\renewbibmacro{in:}{}
\AtBeginBibliography{\small}

\title{$p$-Adic asymptotic subalgebra enumeration}
\date{\today}
\author{Tomas Reunbrouck}

\address{ Fakult\"at f\"ur Mathematik,
Universit\"at Bielefeld,
33615 Bielefeld, Germany}
\email{treunbrouck@math.uni-bielefeld.de}

\keywords{Nilpotent algebras, poles, zeta functions, asymptotics}
\subjclass[2020]{
 11M41
, 11D88  
, 20F69 
, 16W50 
, 45M05  
, 17B30 
}

\thanks{The author was supported by the Deutsche Forschungsgemeinschaft (DFG, German Research Foundation) – Project-ID~491392403 – TRR~358.}

\begin{document}

\begin{abstract}
    We introduce the notion of $p$-adic asymptotics, or $p$-asymptotics, to the context of finite-index subgroup and subalgebra enumeration. For finitely generated groups and finite-dimensional algebras, we connect these asymptotics with the poles of their associated local zeta functions. Our two main results establish the smallest real pole for local zeta functions associated with residually nilpotent algebras, as well as its simplicity and residue whenever this algebra is graded. We thereby provide proof to parts of two conjectures raised by Rossmann \cite{TR14} and give a precise description of the $p$-asymptotic behaviour inside these algebras.
\end{abstract}

\maketitle

\tableofcontents

\section*{Introduction}
\addtocontents{toc}{\protect\setcounter{tocdepth}{1}}

\subsection*{Zeta functions of groups and algebras}

Given a group $G$ and a positive integer $n$, let $a_n(G)$ denote the number of subgroups $H \leq G$ such that $|G:H| = n$. Similarly, let $a^\lhd_{n}(G)$ denote the number of normal subgroups of index $n$. It is a well-known fact that if $G$ is a finitely generated group, then $a^*_n(G) < \infty$ for all $n$, where $*$ indicates either subgroups or normal subgroups. This fact motivated Grunewald, Segal, and Smith \cite{gss} to define the Dirichlet series associated with the finite-index subgroups and normal subgroups of $G$:
$$ \zeta^\leq_G(s) \coloneq  \sum_{n=1}^\infty a_{n}(G) \, n^{-s} \ \text{ and } \ \zeta^\lhd_G(s) \coloneq  \sum_{n=1}^\infty a^\lhd_{n}(G) \, n^{-s} $$
If $G$ is also nilpotent, then these Dirichlet series allow an Euler product:
$$ \zeta^*_{G}(s) = \prod_{p \text{ prime}} \sum_{i=0}^\infty a^*_{p^i}(G) (p^{-s})^i $$
This is a consequence of the fact that finite nilpotent groups are the direct product of their Sylow subgroups. The Dirichlet series $\zeta^*_{G}(s)$ is called the global (normal) subgroup zeta function of $G$. We refer to the factors of the Euler product $\zeta^*_{G,p}(p^{-s}) \coloneq \sum_{i=0}^\infty a^*_{p^i}(G) (p^{-s})^i$ as the local (normal) subgroup zeta functions of $G$. Note that each of these can be seen as a power series in a new variable $p^{-s}$, which we will often write as $t$.

Groups which are torsion-free, in addition to being finitely generated and nilpotent, are called $\T$-groups and they are of special interest as the problem of counting their subgroups can essentially be `linearized'. In a nutshell, this means that by a correspondence theorem due to Malcev (see, for example \cite[Chapter 6]{segal_polycyclic}) for any given $\T$-group $G$, there exists a Lie algebra $L(G)$ over $\Z$, such that for all but finitely many primes $p$, to every (normal) subgroup $H \leq G$ of index $p^i$ naturally corresponds a subalgebra (resp. ideal) $M \leq L(G)$ of that same index, and vice versa. To make this more precise, we are naturally lead to define zeta functions of algebras as well.

When we say that $L$ is an $n$-dimensional algebra over $\Z$, we mean that $L$ is a free $\Z$-module of rank $n$ endowed with a bilinear `multiplication' map \textemdash \ not necessarily commutative or even associative. The algebra $L$ is a Lie algebra if this multiplication satisfies the axioms of a Lie bracket: anticommutativity and the Jacobi identity. We call $L$ abelian, when the multiplication is trivial. A subalgebra of $L$ then, is an additive subgroup of $L$ which is closed under multiplication of its elements. An ideal is an additive subgroup closed under left- and right-multiplication by any element of $L$. Now we let $a^\leq_m(L)$ and $a^\lhd_m(L)$ denote the number of subalgebras and ideals of index $m$. Then, in a similar vein, we define the global and local subalgebra zeta functions of $L$ as
$$ \zeta^*_{L}(s) \coloneq \sum_{m=1}^\infty a^*_{m}(L) \, n^{-s} \ \text{ and } \  \zeta^*_{L,p}(t) \coloneq \sum_{i=0}^\infty a^*_{p^i}(L) \, t^i $$
Once again, we have an Euler product $ \zeta^*_{L}(s) = \prod_{p \text{ prime}} \zeta^*_{L,p}(p^{-s}) $, this time as a consequence of the Chinese Remainder Theorem. For $\T$-groups, the Euler product relied on nilpotency, but for algebras we do not need to impose this restriction. Thus, our motivation originating with $\T$-groups notwithstanding, we are interested in local subalgebra and ideal zeta functions outside of the class of nilpotent algebras as well. The scope of subalgebra zeta functions hence being much broader than that of $\T$-groups, we will mostly focus on algebras for the remainder of this text.

\subsection*{Notation}

Let $\N$ denote the set of non-negative integers $\{0, 1, 2, \dots\}$ and $[n] := \{1, \dots , n\}$. The $p$-adic integers and numbers are written as $\Z_p$ resp. $\Q_p$. For $a \in \Z_p$ and $q$ a $p$-power, the $q$-adic valuation is denoted $v_q(a) \coloneq \max \{ m \in \N \, | \, q^m \text{ divides } \,  a \}$. For a sequence $(a_i)_{i\in \N}$ over $\Q_p$ and $a \in \Q_p$, we denote $p$-adic convergence by $a_i \underset{p}{\longrightarrow} a$. We say that $(a_i)_{i\in \N}$ behaves $p$-asymptotically like another sequence $(b_i)_{i \in \N}$, denoted $a_i \underset{p}{\sim} b_i$, if $a_i - b_i \underset{p}{\longrightarrow} 0$.

\subsection*{Setup of the paper}

Given an $n$-dimensional $\Z$-algebra $L$, it is a well-known fact that for any prime $p$, the subalgebras of $p$-power index are in bijection with those of $L_p \coloneq  L \otimes \Z_p$. The $\Z_p$-algebra $L_p$ is additively isomorphic to $\Z_p^n$ endowed with a $\Z_p$-linear multiplication induced by the $\Z$-linear multiplication of $L$. Note that since $L_p$ is a $\Z_p$-algebra, all its subalgebras are of $p$-power index, and that therefore there is no difference between the local and global zeta functions. Thus we can simply write $\zeta^*_{L_p}(t)$, without ambiguity, and state: $\zeta^*_{L,p}(t) = \zeta^*_{L_p}(t)$. 

One can also consider such $\Z_p$-algebras in general, not necessarily inherited from some initial global object, as for $L$ and $L_p$. In this paper we study such algebras in their own right, detached from an initial global context. In fact, we need not restrict ourselves to $\Z_p$-algebras. All results in this paper are stated for algebras over any compact discrete valuation ring (or DVR) of characteristic $0$. Note that this class of DVRs consists precisely of the rings of integers of $p$-adic fields. 

Throughout, $\O$ denotes a compact DVR of characteristic $0$, $\P$ its maximal ideal, and $\K$ its field of fractions. Its valuation is $\nu$ and $\pi \in \P \setminus \P^2$ is a choice of uniformizing parameter. Note that $\P = \pi \O$. The residue field is $\F_q$ with characteristic $p$. Now, an $n$-dimensional $\O$-algebra $A$ is defined as the $\O$-module $\O^n$ endowed with an $\O$-linear multiplication. Subalgebras are $\O$-submodules of $\O^n$ such that they are closed under multiplication. Ideals are subalgebras which are closed under left- or right-multiplication by any element of $A$. Any subalgebra or ideal is therefore of index $q^i$ for some $i \in \N$, and we let $a_{q^i}(A)$ resp. $a^\lhd_{q^i}(A)$ denote their number. The subalgebra and ideal zeta functions of $A$ are, as before,
$$ \zeta^*_{A}(t) \coloneq  \sum_{i=0}^\infty a^*_{q^i}(A) t^i. $$
We will write $H \leq_{q^i} A$ if $H$ is an index-$q^i$ subalgebra of $A$. 

\subsection*{Rationality and poles}

Given a finite-dimensional $\O$-algebra $A$, its zeta function is rational, by Grunewald, Segal, and Smith \cite{gss}. That is, there exist integer polynomials $P(T), Q(T) \in \Z[T]$ such that
$$ \zeta^*_{A}(t) = \frac{P(t)}{Q(t)}\, . $$
Du Sautoy \cite{duSau_uniformity} showed  that there exist an $e\in\N$, nonnegative integers $d_1, \dots, d_e$, and distinct pairs of non-negative integers $(a_1,b_1),\dots,(a_e,b_e)$, such that if $P(T)/Q(T)$ is reduced
\begin{equation}\label{Q(T)=prod}
      Q(T) \ \text{ divides } \  \prod_{j=1}^e (1-q^{a_j}T^{b_j})^{d_j+1} .
\end{equation}
In particular, the poles of the local zeta function are of the form 
\begin{equation}\label{local_poles_equation}
    t = \mu \ q^{-a_j/b_j} \ \text{ or, equivalently } \ s = \frac{a_j}{b_j} + \frac{2 \pi k}{b_j \log q}i 
\end{equation}
where $\mu$ denotes any $b_j$-th root of unity and $k \in \Z$. By taking a partial fraction decomposition of $\zeta_A^*(t)$ and considering the resulting  sum of geometric series, one shows that the coefficients can be seen as a finite sum of exponential functions. More precisely, there exist quasi-polynomials $R_j(x)$ over $\Q$ such that:
\begin{equation}\label{a=sum_exp}
 a^*_{q^i}(A) = \sum_{j=1}^e R_j(i) \, (q^{a_j/b_j})^i \, . 
\end{equation}
The degree of $R_j$ is $d_j$ and its period divides $b_j$. 
Equation \ref{a=sum_exp} clearly shows how the poles of the local zeta function inform us about the asymptotic behaviour of the coefficients $a^*_{q^i}(L)$. Indeed, assume without loss of generality that the pairs $(a_j,b_j)$ are numbered such that the numbers $\frac{a_j}{b_j}$ are in non-decreasing order, then the pole at $s = \frac{a_e}{b_e}$, being the largest real pole, dominates the asymptotics:
\begin{equation*}
    a^*_{q^i}(A) \sim  i^{d_e} \exp\left( \frac{a_e}{b_e} i  \log q\right) \ \text{ as } i \to \infty .
\end{equation*}
However, from a $p$-adic point of view it is the smallest real pole at $s=\frac{a_1}{b_1}$ that dominates.
\begin{equation}\label{a_to_smallest_pole_eq}
    a^*_{q^i}(A) \underset{p}{\sim}  R_1(i) \, q^{(a_1/b_1)i}  \ \text{ as } i \to \infty .
\end{equation}

From now on, we will refer to a pole, whether in terms of the variable $t$ or $s$, as big or small, depending on whether the associated fraction $\frac{a_j}{b_j}$ is big or small.

\begin{example}\label{abelian_example}
    In the abelian case (i.e. trivial multiplication) we have
    $$ \zeta^\lhd_{\O^n}(t) =  \zeta_{\O^n}(t) = \frac{1}{(1-t)(1-qt)\dots (1-q^{n-1}t)} \,.$$
    For any pair of non-negative integers $a, b$, let $\binom{a}{b}_q \coloneq  \frac{(1-q^a)\dots(1-q^{a-b+1})}{(1-q)\dots(1-q^b)}$ denote the Gaussian binomial coefficient. Then
    $$ a_{q^i}(\O^n) = \binom{n-1+i}{n-1}_q = \frac{1}{\prod_{a=1}^{n-1} (1-q^a)} \sum_{c=0}^{n-1} q^{ci} \left( (-1)^c \binom{n}{c}_q q^{c(c+1)/2} \right). $$
    Hence, asymptotically
    $$ a_{q^i}(\O^n) \sim   q^{(n-1)i} . $$
    More interestingly for our purposes, however, is that $p$-asymptotically, we have
    \begin{equation}\label{abelian_example_equation}
        a_{q^i}(\O^n) \underset{p}{\longrightarrow} \frac{1}{(1-q)\dots(1-q^{n-1})} \eqcolon \binom{\infty}{n-1}_q \, .
    \end{equation}
\end{example}

\begin{example}[\cite{gss}, Proposition 8.1]\label{H_example}
    We define the Heisenberg algebra $\h(\O)$ over $\O$ as the Lie algebra consisting of strictly upper triangular $(3\times3)$-matrices endowed with the usual Lie bracket. $$\zeta_{\h(\O)}(t) = \frac{1-q^3t^3}{(1-t)(1-qt)(1-q^2t^2)(1-q^3t^2)}\, .$$ 
    A straightforward calculation shows that
    \begin{align*}
    a_{q^i}(\h(\O)) &= \frac{1}{(1-q)(1-q^2)}1^i + \left[ \frac{-3q(q+1) }{2(1-q)(1-q^2)} \right] q^i + 
     \left[ \frac{-q^2-q+2 }{2(1-q)(1-q^2)} \right] (-q)^i \\
    &+ \left[ \frac{ q^2(1+q) + q^{3/2} + q^{5/2} }{2(1-q)(1-q^2)} \right] (q^{3/2})^i + \left[ \frac{ q^2(1+q) - q^{3/2} + q^{5/2} }{2(1-q)(1-q^2)} \right] (-q^{3/2})^i \, .
\end{align*}

\noindent Thus, asymptotically we have
$$ a_{q^i}(\h(\O)) \sim q^{\frac{3}{2}i} \, . $$
and $p$-asymptotically,
$$ a_{q^i}(\h(\O)) \underset{p}{\longrightarrow} \frac{1}{(1-q)(1-q^2)}\, . $$
The latter expression matches Equation \ref{abelian_example_equation}, the abelian $p$-asymptotic limit for $n=3$. As we will see, this is not a coincidence.
\end{example}

\subsection*{$p$-Asymptotics}
The previous section should provide ample motivation to investigate the smallest pole of subalgebra and ideal zeta functions. For all known examples, this pole appears at $t=1$ (equivalently, $s=0$). By Equation \ref{local_poles_equation}, this is the smallest pole possible, but it is unclear whether there is always a pole at $t=1$. This was conjectured by Rossmann.

\begin{conjecture}[Rossmann (2014) \cite{TR14}, Conjecture II]\label{pole_exists_conjecture_Rossmann}
    Let $A$ be a finite-dimensional $\O$-algebra. Then $\zeta_{A}(t)$ and $\zeta^\lhd_{A}(t)$ have a pole at $t=1$.
\end{conjecture}

\begin{remark}\label{pole_exists_conj_remark}
    In terms of $p$-asymptotic behaviour, by Expression \ref{a_to_smallest_pole_eq}, the conjecture is equivalent to saying that the zeta function's coefficients do not $p$-adically converge to $0$:
$$ v_q( a^*_{q^i}(A) )  \not\rightarrow  \infty \ \text{ as } \ i \to \infty. $$
\end{remark}

\noindent The first major result of this paper is to answer Conjecture \ref{pole_exists_conjecture_Rossmann} in the affirmative for residually nilpotent algebras $A$. In fact, we will prove the following, stronger statement.

\begin{thmx}\label{1modp_nilpotent_algebras}
    Let $A$ be a finite-dimensional residually nilpotent $\O$-algebra. Then for all $i \in \N$,
    $$ a_{q^i}(A) \equiv 1 \bmod q \ \ \text{ and } \ \ a^{\lhd}_{q^i}(A) \equiv 1 \bmod q $$
\end{thmx} 

\noindent The full conjecture, however, remains open. For nilpotent algebras, Rossmann goes further and conjectures the following.

\begin{conjecture}[Rossmann (2014) \cite{TR14}, Conjecture IV]\label{residue_conjecure_Rossmann}
    Let $A$ be an $n$-dimensional nilpotent $\O$-algebra. Then $\zeta_{A}(t)$ and $\zeta^\lhd_{A}(t)$ have a simple pole at $t=1$ with residue $\binom{\infty}{n-1}_q$.
\end{conjecture}

\noindent The additional requirement that $A$ be nilpotent is essential, as it is easy to find non-nilpotent examples for which the conclusion of the conjecture does not hold.

\begin{example}\label{Z2_example}
    The subalgebra zeta function of $\Z_p^2$ equipped with component-wise multiplication can be found in Snocken's PhD-thesis \cite[Proposition 7.15]{snocken_thesis}:
    $$ \zeta_{(\Z_p^2,\, \cdot\, )}(t) = \frac{(1-t^2)^2}{(1-t)^3(1-pt^3)}\ . $$
    Its residue at $t=1$ is $\frac{4}{1-p}$. We explain this discrepancy in Example \ref{Z2_revisited_example}.
\end{example}

\noindent In essence, Conjecture \ref{residue_conjecure_Rossmann} states that at $t=1$ the zeta functions of nilpotent algebras behave the same as the zeta function of the abelian algebra of the same dimension. From the perspective of $p$-asymptotics Conjecture \ref{residue_conjecure_Rossmann} is equivalent to saying that the zeta function's coefficients converge to the conjectured residue:
$$ a_{q^i}(A) \  \underset{p}{\longrightarrow} \ \binom{\infty}{n-1}_q \ \text{ as } \ i \to \infty.$$

\noindent Note that $\binom{\infty}{n-1}_q$ is the $p$-adic limit for $a_{q^i}(\O^n)$, as we saw in Example \ref{abelian_example}. 

 Conjecture \ref{residue_conjecure_Rossmann} has drawn considerable interest since its formulation: see  \cite[Theorem 1.1, Corollary 3.4]{voll_ideal_class2}, \cite[ Conjecture 1.11, Remark 1.13]{voll_submod_nilp_endo}, \cite[Conjecture E]{rossmann_submodules_enomorphism}, \cite[Section 9.3]{rossmann_computing_II}. Variants and follow-up questions have been considered: \cite[Conjecture 6.11]{lee_voll_graded_ideals}, \cite[Question 10.3]{rossmann_computing_local}. However, so far, progress towards answering Conjecture \ref{residue_conjecure_Rossmann} has been limited. For an infinite class of nilpotent algebras, the following two theorems answer Rossmann's Conjecture \ref{residue_conjecure_Rossmann} in the positive.

\begin{thmx}\label{class2_residue_theorem}
    Let $A$ be a class-$2$ nilpotent $n$-dimensional $\O$-algebra. Then $\zeta_A(t)$ and $\zeta^\lhd_A(t)$ have a simple pole at $t=1$ with residue $\binom{\infty}{n-1}_q$. 
\end{thmx}

\noindent For $\zeta_A(t)$, we can go further and prove Conjecture \ref{residue_conjecure_Rossmann} in the much larger class of graded algebras. For any commutative ring $R$, an $R$-algebra $A$ is called graded, if it can be written as a direct sum of $R$-modules:  $A = \oplus_{j=1}^\infty A_j$ such that $A_i \cdot A_j \subseteq A_{i+j}$ for all $i,j \in \N$.

\begin{thmx}\label{grading_residue_theorem}
    Let $A$ be an $n$-dimensional graded $\O$-algebra. Then $\zeta_{A}(t)$ has a simple pole at $t=1$ with residue $\binom{\infty}{n-1}_q$.
\end{thmx}

\begin{remark}
    Class-$2$ nilpotent algebras are automatically graded. Therefore, Theorem \ref{grading_residue_theorem} subsumes Theorem \ref{class2_residue_theorem} when it comes to $\zeta_A(t)$. 
\end{remark}

\noindent Although graded algebras only form a subclass of nilpotent algebras, all nilpotent algebras for which the subalgebra zeta function is known are graded. When it comes to $\zeta_A(t)$ therefore, Theorem \ref{grading_residue_theorem} covers all known data supporting Conjecture \ref{residue_conjecure_Rossmann}. We will prove Theorem \ref{grading_residue_theorem} in Section \ref{graded_algebras_section}.


\section{Residually nilpotent algebras}

\subsection{Proof of Theorem \ref{1modp_nilpotent_algebras} }

Given a finite-dimensional residually nilpotent $\O$-algebra $A$, we want to prove that $a_{q^i}(A) \equiv 1 \bmod q$ and $a^\lhd_{q^i}(A) \equiv 1 \bmod q$. We will only show the former; the proof of the latter is analogous. To this end, we introduce the following notation: $b_{q^j}(K)$ denotes the number of \textit{superalgebras} $H$ of $K$ with co-index $q^j$ \textemdash \ that is:
$$ b_{q^j}(K) \coloneq \# \left\{ H \, \middle| \, K \leq_{q^j} H \leq A \right\}. $$

\begin{lemma}\label{suma=sumb}
    Let $A$ be an $n$-dimensional $\O$-algebra. Then for any $i\in \N$,
    \begin{equation}\label{suma=sumb_eq}
        \sum_{H \leq_{q^i} A } a_{q}(H) = \sum_{K \leq_{q^{i+1}} A} b_q(K) .
    \end{equation}
\end{lemma}
\begin{proof} \
    \vspace{-0.25cm}
    \begin{align*}
        \sum_{H \leq_{p^i} A } a_{q}(H) &= \sum_{H \leq_{q^i} A }  \# \left\{ K \, \middle| \, K \leq_q H \leq_{q^i} A \right\} \\
        &=  \# \left\{ (K,H) \, \middle| \, K \leq_q H \leq_{q^i} A \right\} \\
        &= \sum_{K \leq_{q^{i+1}} A }  \# \left\{ H \, \middle| \, K \leq_q H \leq_{q^i} A \right\} = \sum_{K \leq_{q^{i+1}} A} b_q(K)
    \end{align*}
\end{proof}

\begin{lemma}\label{abp=1}
    If $A$ is a residually nilpotent $\O$-algebra, then in Equation \ref{suma=sumb_eq} we have,
    \begin{itemize}
        \item[a)]\label{ap=1} $a_q(H) \equiv 1 \bmod q$
        \item[b)]\label{bp=1} $b_q(K) \equiv 1 \bmod q$
    \end{itemize}
\end{lemma}

\begin{proof}
    For a), let $H\leq_{q^i}A$ and let $H^2$ denote the subalgebra generated by $H \cdot H \coloneq \{ a \cdot b \ | \ a \in H, b \in H \}$. Since $A$ is residually nilpotent, $H^2 < H$. The multiplication induced on $H/H^2$ is trivial, and therefore $H$ has at least one ideal $I$ of index $q$. Let $K \leq_q H$ be a subalgebra different from $I$ (note that $K \lhd H$ due to residual nilpotency). Then $K\cap I$ is an ideal of $K$ and by the second isomorphism theorem is $K / (I\cap K)$ isomorphic to $(I+K) / I$ which is $H/I$. This implies that $I\cap K$ has index $q$ inside $K$ and therefore index $q^2$ inside $H$. Thus, $H/(I \cap K)$ has size $q^2$ and one easily checks (again, using nilpotency) that it must therefore have either $1$ or $q+1$ proper subalgebras. However, as $H/(I \cap K)$ has at least two \textemdash \ $K/(I\cap K)$ and $I/(I\cap K)$ are \textit{different} subalgebras \textemdash \  we know that $H/(I \cap K)$ must have $q+1$ proper subalgebras. This means that $K$ lies inside a family of $q$ subalgebras $K_1,...,K_q$ of index $q$. If there is another subalgebra $K'\leq_q H$ which does not lie in this family, then the same reasoning for $I \cap K'$ produces a family $K'_1, ..., K'_q$ containing $K'$, which is disjoint from $K$'s family. Repeating this for every subalgebra of index $q$ in $H$ which is not $I$, we see that they lie in equivalence classes of size $q$ defined by the equivalence relation $K \sim K' : \! \iff (I\cap K) = (I\cap K')$. Hence, $a_q(H)$ is equal to $1$ plus $q$ times the number of these equivalence classes, proving that $a_q(H) \equiv 1 \bmod q$.

    For b), let $K \leq_{q^{i+1}}A$. Since $A$ is residually nilpotent, $I_A(K) > K$ (where $I_A(K)$ denotes the idealizer of $K$ inside $A$). Now $I_A(K)/K$ has an ideal $J/K$ of order $q$, since it is a nilpotent $\O$-algebra with order a power of $q$. This implies the existence of an ideal $J \lhd I_A(K)$ containing $K$ such that $[J:K] = q$. Note that though $J \leq A$, it is not necessarily the case that $J\lhd A$. We have proven that $b_q(K)$ is at least $1$. If there is another subalgebra $H \leq_{q^i} A$ containing $K$, then $K\lhd H$ and thus $H \leq I_A(K)$ and $J+H \leq I_A(K)$. By the second isomorphism theorem is $(J+H)/J$ isomorphic to $H/(J\cap H)$ which is $H/K$. Thus $(J+H)/K$ has order $q^2$, and because it has at least two proper subalgebras, it must have precisely $q+1$ proper subalgebras. This means that $H$ is a member of a family $H_1,...,H_q$ all of which contain $K$ as a subalgebra of index $q$. Moreover, given a another superalgebra $H' \geq_q K$ outside of this family, then the same reasoning can be repeated for $JH'$, producing a new family $H'_1,...,H'_q$ disjoint from the family of $H$. Overall, the index-$q$ superalgebras of $K$ not equal to $J$ come in equivalence classes of size $q$, defined by the equivalence relation $H \sim H' :\! \iff J+H=J+H'$. As before, this implies that $b_q(K) \equiv 1 \bmod q$.
\end{proof}

\begin{proof}[Proof for Theorem \ref{1modp_nilpotent_algebras}]
    Reducing the equation in Lemma \ref{suma=sumb} modulo $q$, Lemma \ref{abp=1} tells us that
    $$ \sum_{H\leq_{q^i}A} 1 \equiv \sum_{K\leq_{q^{i+1}}A} 1 \mod q, $$
    which is the same as
    $$ a_{q^{i+1}}(A) \equiv a_{q^i}(A) \mod q. $$
    Since $a_{q^0}(A) = 1$, we know that $a_{q^{i}}(A) \equiv 1 \bmod q$ for all $i\in\N$. 
\end{proof}

 \noindent Theorem \ref{1modp_nilpotent_algebras} implies that Conjecture \ref{pole_exists_conjecture_Rossmann} is true for residually nilpotent algebras.
\begin{corollary}
    Let $A$ be a finite-dimensional, residually nilpotent $\O$-algebra. Then $\zeta_{A}(t)$ and $\zeta^{\lhd}_{A}(t)$ have a pole at $t=1$.
\end{corollary}

\begin{proof}
    As mentioned in Remark \ref{pole_exists_conj_remark}, $\zeta_{A}(t)$ has a pole at $t=1$ if and only if $v_p(a_{q^i}(A)) \not\to \infty$ as $i \to \infty$. Since by Theorem \ref{1modp_nilpotent_algebras}, $a_{q^i}(L) \equiv 1 \bmod q$ for all $i\in \N$, the sequence $v_p(a_{q^i}(A))$ is constantly $0$. The same goes for $\zeta^{\lhd}_{A}(t)$. 
\end{proof}

\begin{remark}\label{pA_res_nilp_remark}
    For any finite-dimensional $\O$-algebra $A$, the algebra $\pi A$ is residually nilpotent. Therefore, for any such $A$, $a_{q^i}(\pi A) \equiv 1 \bmod q$ and $\zeta_{\pi A}(t)$ has a pole for $t=1$.
\end{remark}

\subsection{Proof of Theorem \ref{class2_residue_theorem}}

Let $A$ be an $n$-dimensional $\O$-algebra. As an $\O$-module, $A$ is isomorphic to $\O^n$. Thus, it makes sense to define $a^{\not\leq}_{q^i}(A)\coloneq a_{q^i}(\O^n) - a_{q^i}(A)$ as the number of sublattices of $A$ that are not subalgebras; that is, submodules of $\O^n$ which are not closed under the multiplication of $A$. Similarly, $a^{\not\lhd}_{q^i}(A)\coloneq a_{q^i}(\O^n) - a^\lhd_{q^i}(A)$ denotes the number of sublattices of $A$ that are not ideals of $A$. We will prove the following.
\vspace{-0.1cm}
\begin{theorem}\label{nonsubalgebras_to_0_theorem}
    Let $A$ be a class-$2$ nilpotent $n$-dimensional $\O$-algebra. Then 
    $$v_q\left(a^{\not\leq}_{q^i}(A)\right) \longrightarrow \infty \ \text{ and } \ v_q\left(a^{\not\lhd}_{q^i}(A)\right) \longrightarrow \infty  \quad \text{ as } \ i \to \infty .$$
\end{theorem}
\noindent The centre of $A$ is defined as $Z(A) \coloneq  \{ z\in A \ | \ \forall a \in A: a\cdot z = z \cdot a = 0 \}$. Note that the centre is a free $\O$-module. We identify $Z(A)$ with $\O^b$ and $A/Z(A)$ with $\O^a$ where $n = a+b$. Then we can view the multiplication of $A$ as an $\O$-linear map $\O^a \times \O^a \to \O^b$. To any sublattice $H\leq_{q^i} A$ we can associate a pair of sublattices $\Lambda_1(H) \leq \O^a , \Lambda_2(H) \leq \O^b$ defined by
    $$ \frac{\Lambda_1(H)}{H\cap Z(A)}  \coloneq  \frac{H}{H\cap Z(A)}, \quad \Lambda_2(H) \coloneq  H \cap Z(A).$$

\noindent We write $i = i_1(H) + i_2(H)$ where $q^{i_1(H)}$ and $q^{i_2(H)}$ are the indices of $\Lambda_1(H)$ and  $\Lambda_2(H)$ in $\O^a$ and $\O^b$, respectively. Such a sublattice $H$ is then a subalgebra, if $(\Lambda_1(H))^2 \leq \Lambda_2$, where $(\Lambda_1(H))^2$ denotes the sublattice of $\O^b$ generated by $\Lambda_1(H) \cdot \Lambda_1(H) \coloneq  \{x\cdot y \ | \ x, y\in \Lambda_1(H) \}$. Similarly, $H$ is an ideal if $\Lambda_1(H) \cdot \O^a \leq \O^b$ and $ \O^a \cdot \Lambda_1(H)  \leq \O^b$. Given a pair of sublattices $\Lambda_1 \leq \O^a, \Lambda_2 \leq \O^b$, the set of sublattices $H \leq A$ for which $\Lambda_1(H) = \Lambda_1$ and $\Lambda_2(H) = \Lambda_2$ is in bijection with $(\O^b / \Lambda_2)^a$. Therefore, the map $H \mapsto(\Lambda_1(H), \Lambda_2(H))$ is $q^{i_2(H)a}$-to-one. Setting 
\vspace{-0.2cm}
\begin{align*}
    \qquad \ c^\leq_{i_1,i_2}(A) &\coloneq \{ (\Lambda_1, \Lambda_2) \ | \ \Lambda_1 \leq_{q^{i_1}} \O^a, \, \Lambda_2 \leq_{q^{i_2}} \O^b, \Lambda_1^2 \leq \Lambda_2 \}, \\
    c^\lhd_{i_1,i_2}(A) &\coloneq \{ (\Lambda_1, \Lambda_2) \ | \ \Lambda_1 \leq_{q^{i_1}} \O^a, \, \Lambda_2 \leq_{q^{i_2}} \O^b, \Lambda_1 \cdot \O^a \leq \Lambda_2, \O^a \cdot \Lambda_1 \leq \Lambda_2 \},
\end{align*}
then
\vspace{-0.2cm}
$$ \ \ \zeta_{A}(t) = \sum_{i_2=0}^\infty (q^at)^{i_2} \sum_{i_1=0}^{\infty} t^{i_1} \# c^\leq_{i_1,i_2}(A)
    \ \text{ and } \ 
    \zeta^\lhd_{A}(t) = \sum_{i_2=0}^\infty (q^at)^{i_2} \sum_{i_1=0}^\infty \#c^\lhd_{i_1,i_2}(A). $$

\noindent Put another way,
\begin{equation}\label{split_aqi_eq}
    a^\leq_{q^i}(A) = \sum_{i_2 = 0}^i q^{a \cdot i_2 } \#c^\leq_{i-i_2,i_2}(A) \ \text{ and } \ a^\lhd_{q^i}(A) = \sum_{i_2 = 0}^i q^{a \cdot i_2 } \#c^\lhd_{i-i_2,i_2}(A).
\end{equation}

\noindent Let $c^{\not\leq}_{i_1, i_2}(L)$ and $c^{\not\lhd}_{i_1, i_2}(L)$ denote the sets of pairs $(\Lambda_1, \Lambda_2)$ with $\Lambda_{1} \leq_{q^{i_1}} \O^a$ and $\Lambda_{2} \leq_{q^{i_2}} \O^b$ such that $(\Lambda_1, \Lambda_2) \notin c^{\leq}_{i_1,i_2}$ and $(\Lambda_1, \Lambda_2) \notin c^{\lhd}_{i_1,i_2}$, respectively. Then,
\begin{equation}\label{nonobjects_equation}
     a^{\not \leq }_{q^i}(A) = \sum_{i_2 = 0}^i q^{a \cdot i_2 } \#c^{\not\leq}_{i-i_2,i_2}(A)
\quad \text{ and } \     a^{\not\lhd}_{q^i}(A) = \sum_{i_2 = 0}^i q^{a \cdot i_2 } \#c^{\not\lhd}_{i-i_2,i_2}(A).
\end{equation}

\noindent In order to prove Theorem \ref{nonsubalgebras_to_0_theorem}, we show that for any $i_2 \in \N$, the valuation of the corresponding summand converges to $\infty$ as $i \to \infty$:
\begin{equation}\label{nonobjects_convergence}
    a \cdot i_2 + v_q(\# c^{\not\leq}_{i-i_2,i_2}(A)) \to \infty \ \text{ and } \ a \cdot i_2 + v_q(\# c^{\not\lhd}_{i-i_2,i_2}(A)) \to \infty.
\end{equation}

\noindent Define $h(\Lambda_1) \coloneq  \max \{ m \ | \ \Lambda_1 \leq \pi^m \O^a \}$. This number is called the homothety of $\Lambda_1$ and we will come back to it in the next subsection. For now, the following suffices.

\begin{lemma}\label{h<i_2}
    If $(\Lambda_1,\Lambda_2) \in c^{\not \leq}_{i_1,i_2}(A)$, then $h(\Lambda_1) < \frac{1}{2} i_2 $. If $(\Lambda_1,\Lambda_2) \in c^{\not \lhd}_{i_1,i_2}(A)$, then 
    $h(\Lambda_1) < i_2. $
\end{lemma}

\begin{proof}
    Since $[\O^b : \Lambda_2] =   q^{i_2}$, $\pi^{i_2} \O^b \leq \Lambda_2$. If $h(\Lambda_1) \geq \frac{1}{2} i_2 $, then $\Lambda_1^2 \leq \pi^{i_2}\O^b \leq \Lambda_2$, which means that $(\Lambda_1,\Lambda_2) \in c^{\leq}_{i_1,i_2}(A)$. 
    Similarly, if $h(\Lambda_1) \geq i_2$, then $\O^a \cdot \Lambda_1 \leq \pi^{i_2} \O^b \leq \Lambda_2$ and $ \Lambda_1 \cdot \O^a \leq \pi^{i_2} \O^b \leq \Lambda_2$, which means that $(\Lambda_1,\Lambda_2) \in c^{\lhd}_{i_1,i_2}(A)$.
\end{proof}

\noindent Now we define an action by $\O$ on the set of sublattices of $\O^a$. Let $\Lambda_1 \leq \O^a$, then we choose a basis for $\Lambda_1$ denoted by $y_1, \dots , y_a$. We write $y_j = \pi^{\lambda_j} x_j$, where $\lambda_j = \nu(y_j)$ and $x_j$ is primitive. The valuation $\nu$ of a vector over a DVR is defined as the minimum of the valuations of its components, and that a vector is called primitive if at least one of its components has valuation $0$ - i.e. is invertible. Note that $x_1,\dots,x_a$ form a basis for $\O^a$. 

Assume, without loss of generality, that the $\lambda_j$ are in non-increasing order. They are the elementary divisors of $\Lambda_1$ and we will revisit them in Section \ref{graded_algebras_section}. Here, we simply note that $\lambda_a = h(\Lambda_1)$. Consider the following action by $\O$ on $\Lambda_1$: for $\mu \in \O$, define $\Lambda^\mu_1$ as the lattice generated by $y_1, \dots,y_{a-1}$, and $ y_a + p^{i_2+1} \mu \cdot x_1$. 

\begin{lemma}
    For all $\mu \in \O$, $(\Lambda^\mu_1, \Lambda_2) \in c^{*}_{i_1,i_2}(A)$ if and only if $(\Lambda_1,\Lambda_2) \in c^{*}_{i_1,i_2}(A)$.
\end{lemma}

\begin{proof}
    Observe that 
    $(\Lambda^\mu_1)^2 \leq \Lambda^2_1 + \pi^{i_2} \O^b$, 
    $\O^a \cdot \Lambda^\mu_1 \leq \O^a \cdot \Lambda_1 + \pi^{i_2} \O^b$ and
    $ \Lambda^\mu_1 \cdot \O^a \leq \Lambda_1 \cdot \O^a + \pi^{i_2} \O^b$. Since $\pi^{i_2} \O^b \leq \Lambda_2$, it follows that $\O^a \cdot \Lambda^\mu_1 \leq \Lambda_2$ if and only if $\O^a \cdot \Lambda_1 \leq \Lambda_2$. Similarly, $\Lambda^\mu_1 \cdot \O^a \leq \Lambda_2$ if and only if $\Lambda_1 \cdot \O^a \leq \Lambda_2$.
\end{proof}

\begin{proof}[Proof of Theorem \ref{nonsubalgebras_to_0_theorem}]

    We prove for every $i_2 \in \N$, that for the valuation of the corresponding summand in Equation \ref{nonobjects_equation} the following holds:
    \begin{equation}\label{ai2+v}
        a \cdot i_2 + v_q\left( \# c^{\not\lhd}_{i-i_2,i_2} (A) \right) \geq \frac{1}{a} i-1.
    \end{equation}
    Choose $(\Lambda_1,\Lambda_2) \in c^{\not\lhd}_{i-i_2,i_2} (A)$. Note that $\lambda_a = h(\Lambda_1)$ and thus by Lemma \ref{h<i_2}, $\lambda_a \leq i_2$. At the same time, $\lambda_1 \geq (i-i_2)/a$. Otherwise, $\sum_j \lambda_j \leq \sum_j \max_j\{\lambda_j\} = a\cdot  \lambda_1 < i-i_2$. But $\sum_{j} \lambda_j = i-i_2$, by assumption.

     From now on, we assume that $a>1$. If not, then $\lambda_1 = \lambda_a$, which means that $i_2 \geq (i-i_2)/a$. Equivalently, $(a+1)i_2 \geq i$, and therefore $a \cdot i_2 + v_q\left( \# c^{\not\lhd}_{i-i_2,i_2} (A) \right) \geq a \cdot i_2 \geq \frac{a}{a+1} i$, which proves the inequality \ref{ai2+v}.

    Consider the action $\Lambda_1 \mapsto \Lambda^\mu_1$ by $\mu \in \O$. If $\lambda_1 > i_2$, the orbit size of this action is $q^{\lambda_1 - i_2-1}$. Therefore, in that case 
    
    \begin{align*}
            a \cdot i_2 + v_q\left( \# c^{\not\lhd}_{i-i_2,i_2} (A) \right) &= \lambda_1 - i_2-1 + a\cdot i_2 \\
            &\geq \frac{i-i_2}{a} + (a-1)i_2-1 \\
        &= \frac{i}{a} + \frac{(a-1)a-1}{a}i_2 -1\geq \frac{i}{a} -1 \, .
    \end{align*}
    If $\lambda_1 \leq i_2$, then $i_1 < a\cdot i_2 $ which means that $i < (a+1)i_2$. Thus, in this case $a \cdot i_2 + v_q\left( \# c^{\not\lhd}_{i-i_2,i_2} (A) \right) \geq a\cdot i_2 \geq \frac{a}{a+1} i $. This proves the inequality \ref{ai2+v}, which implies the convergence \ref{nonobjects_convergence}. As a consequence $v_q\left( a^{\not\lhd}_{q^i}(A) \right) \to \infty$ for $i\to \infty$, as Theorem \ref{nonsubalgebras_to_0_theorem} claims. The proof for $v_q\left( a^{\not\leq}_{q^i}(A) \right) \to \infty$ for $i\to \infty$ is identical.
\end{proof}

\begin{corollary}[Theorem \ref{class2_residue_theorem}]
    For $A$ a class-$2$ nilpotent $\O$-algebra
    $$ a^*_{q^i}(A) \underset{p}{\longrightarrow} \binom{\infty}{n-1}_q \quad \ \text{ as } \ i \to \infty. $$
\end{corollary}

\subsection{Homothety and weights}\label{homothety_section}

The existence and conjectured residue of the pole in $t=1$ for $\zeta_A(t)$ for a residually nilpotent algebra $A$ can be viewed in terms of homothety classes. Recall that we denote $\O$'s field of fractions by $\K$. Given a sublattice $\Lambda \leq \O^n$ of rank $n$, we define its homothety class as $[\Lambda] \coloneq \{ \Lambda' \leq \K^n \, | \, \exists m \in \Z : \pi^m \Lambda = \Lambda' \}$. If such a sublattice $\Lambda$ is closed under the multiplication of $A$ \textemdash \ in other words, if $\Lambda$ is a subalgebra of $A$ \textemdash \ then $\pi^m\Lambda$ is also a subalgebra for any $m \in \N$. This means that every homothety class $[\Lambda]$ contains a unique \textit{maximal} subalgebra $\Lambda^{\max} \leq A$, such that there exists no $\Lambda' \in [\Lambda]$ for which $\pi\Lambda' = \Lambda$ and $\Lambda' \leq A$. Let $a^{\max}_{q^i}(A)$ denote the number of maximal subalgebras of $A$ of index $q^i$. Note that $a^{\max}_{q^i}(A) = a_{q^i}(A) - a_{q^{i-n}}(A)$. Therefore, we define $ \zeta^{\max}_{A}(t) \coloneq  \sum_{i=0}^\infty a^{\max}_{q^i}(A)t^i = (1-t^n) \zeta_{A}(t) $.

\subsubsection{Simplicity of the pole at $t=1$}\label{simple_pole_subsubsection}

So far we have proven that if $A$ is an $n$-dimensional, residually nilpotent $\O$-algebra, then $\zeta_{A}(t)$ has a pole in $t=1$. This pole is simple if and only if $(1-t^n)\zeta_{A}(t)$ does not have a pole at $t=1$. This is equivalent to saying that $a^{\max}_{q^i}(A)$ converges $p$-adically to $0$. If this is the case, then the sum $\sum_{i=0}^\infty a^{\max}_{q^i}(A)$ is $p$-adically convergent and its value coincides with the residue at $t=1$.

 Each homothety class $[\Lambda]$ contains a unique \textit{primitive} lattice $\Lambda_0 \leq \O^n$, which is a lattice such that $\pi^{-1}\Lambda_0 \not\leq \O^n$. In \cite{voll_functional}, Voll defines the \textit{weight} of a lattice as the minimal homothety required to turn it into a subalgebra of $A$: $ w(\Lambda) \coloneq \min \{ w \in \N \, | \, \pi^w \Lambda \leq A \} $. Note that for any  lattice $\Lambda$ of index $q^i$, the $i$-th homothety $\pi^i \Lambda$ is certainly a subalgebra of $A$, irrespective of $A$'s multiplicative structure. This means that $w(\Lambda) \leq i(\Lambda)$ for any $\Lambda$. Let $\tilde{a}^w_{q^i}(A)$ denote the number of primitive $\Lambda_0$ with $i(\Lambda_0) = i$ and $w(\Lambda_0) = w$. Then
 $$ (1-t^n)  \zeta_{A}(t) = \sum_{i=0}^\infty a^{\max}_{q^i}(A) t^i = \sum_{\Lambda_0 \text{ primitive}} t^{i(\Lambda_0) + n\cdot w(\Lambda_0)} 
    = \sum_{k=0}^{\infty} c_k(A) \, t^k. $$

\noindent where $ c_k(A) \coloneq \sum_{\substack{ (i,w) \\ i+n\cdot w = k }} \tilde{a}^w_{q^i} (A)$. The pole at $t=1$ is simple if and only if $\left(c_k(A)\right)_{k\in \N}$ converges $p$-adically to $0$. If so, the sum $\sum_{k \in \N} c_{k}(A)$ is $p$-adically convergent and the residue of the pole is equal to $ \frac{1}{n} \sum_{k \in \N} c_{k}(A)$, since

\begin{equation}\label{sumc=nres}
    \sum_{k \in \N} c_{k}(A) = [(1-t^n)  \zeta_{A}(t)]_{t=1} = n [(1-t)  \zeta_{A}(t)] = n \cdot \text{res}_{t=1} \zeta_{A}(t).
\end{equation}

\begin{example}\label{Z2_revisited_example}
    Revisiting Example \ref{Z2_example}, recall that the subalgebra zeta function of $\Z^2_p$ has a simple pole at $t=1$ with residue $\frac{4}{1-q}$. We see this reflected in the $p$-asymptotic behaviour of the coefficients as follows. Firstly, one easily computes that 
    $$ \tilde{a}^w_{p^i}(L) = \begin{cases} 3\varphi(p^w) \ &\text{ for } w < i \\ (p-2)p^{i-1} &\text{ for } w=i \end{cases} $$
    where $\varphi$ denotes the Euler totient function. Thus, if $k=3l+1$ or $k=3l+2$ for some $l\in \N$, then $c_k(L) = \sum_{w=0}^l 3\varphi(p^w) = 3p^{l-1}$. If $k=3l$ (for $l>0$), then $c_k(L) = 3p^{l-1} + (p-2)p^{l-1} = (p+1)p^{l-1}$. Not only do we see that $v_p(c_k(L)) \to \infty$ as $k\to \infty$, we can calculate the cumulative sum:
    \vspace{-0.2cm}
    \begin{align*}
        \sum_{k \in \N} c_k(L) &= 1 +\sum_{l =1}^{\infty} \left(c_{3l}(L) + c_{3l+1}(L) + c_{3l+2}(L) \right) \\
        &= 1 + \sum_{l=1}^{\infty} (p+7)p^{l-1} = \frac{8}{1-p}\ .
    \end{align*}
    This indeed matches the residue (the extra factor $2$ coming from the dimension $n=2$ as in Equation \ref{sumc=nres}).
\end{example} 

\noindent Note that the sum $\sum_{k \in \N} c_{k}(A)$ essentially `counts' all primitive lattices of $\O^n$. That is, its summands are exactly the elements of the set $\left\{\tilde{a}^w_{q^i}(A)\right\}_{(i,w)}$. If we were to sum this set in a different way, we already know its value:
\begin{equation}\label{i_vs_k}
    \sum_{i \in \N} \sum_{w = 0}^i \tilde{a}^w_{q^i}(A) = \sum_{i \in \N} \tilde{a}^w_{q^i}(A) = \sum_{i \in \N} \tilde{a}^w_{q^i}(\O^n) = \binom{\infty}{n-1}_q.
\end{equation}
\noindent This is the proposed residue from Conjecture   \ref{residue_conjecure_Rossmann}. However, we cannot assume that rearranging the sum has no influence on its value. More precisely, we need that
\begin{equation}\label{kiwi_equation}
    \sum_{k \in \N} \left( \sum_{w=0}^{\left\lfloor\frac{k}{n+1}\right\rfloor} \tilde{a}^w_{q^{k-n\cdot w}} (A) \right) = \sum_{i \in \N} \left( \sum_{w=0}^i \tilde{a}^w_{q^i}(A) \right).
\end{equation}

\noindent Despite this Example \ref{Z2_revisited_example}, there are many algebras for which this equation holds. In particular, if we have absolute convergence of $\sum_{(i,w)}\tilde{a}^w_{q^i}(A)$. Then, by Fubini's Theorem, we can rearrange the sums as in Equation \ref{kiwi_equation}.

\begin{theorem}[Fubini ($p$-adic version); see \cite{igusa_intro}, Section 7.2]\label{fubini}
    Let $f: \N \times \N \to \Z_p : (k,l) \mapsto f(k,l) $ be a function such that $\sum_{(k,l) \in \N \times \N} |f(k,l)|_p < \infty$. Then $\sum_{(k,l) \in \N \times \N} f(k,l)$ is well-defined and its value is invariant under any rearranging of the summands.
\end{theorem}

\noindent Therefore, what we have to prove is that for any $M\in \N$ there are only finitely many $(i,w) \in \N \times \N$ such that $v_p(\tilde{a}^w_{q^i}(A)) \leq M$. This is true, in particular, when there exists a $\lambda > \R_{>0}$ such that $v_p(\tilde{a}^w_{q^i}(A)) \geq \lambda \cdot i$ for all $i$ large enough.  In Section \ref{graded_algebras_section} we will prove that this is indeed true when $A$ is graded.

\section{Graded algebras}\label{graded_algebras_section}

\noindent The proof that we have absolute convergence for $\sum_{(i,w)}\tilde{a}^w_{q^i}(A)$ whenever $A$ is graded, is based on a correspondence between subalgebras and solutions to polynomial congruences. We outline this correspondence, borrowing the notation used by Voll in \cite[Section 3.1]{voll_functional}.

Let $\Gamma \coloneq \GL_n(\O)$. There is a one-to-one correspondence between sublattices of $\O^n$ and right cosets $\Gamma M$ for $M\in \text{Mat}_n(\O)$. We consider the lattice corresponding to $M$ as the $\O$-span of the rows of $M$. Note that $\det(M)$ is equal to the index $q^i$ of the corresponding lattice. By the Elementary Divisor Theorem there exists, for any coset $\Gamma M$, a diagonal matrix $D$ of non-increasing $\pi$-powers $\pi^{\lambda_1}, \dots, \pi^{\lambda_n}$, such that $ \Gamma M = \Gamma D \gamma $ for some $\gamma \in \Gamma$. Here, $\lambda = (\lambda_1,\dots, \lambda_n)$ is a partition of $i = \nu(\det(M))$. In the following we will assume that $\gamma = \beta \, C_\sigma$, where $\beta$ is a lower unitriangular matrix (that is, lower-triangular with every diagonal entry equal to $1$), $\sigma \in \sym_n$, and $C_\sigma$ the permutation matrix that acts via right multiplication like $\sigma$ (on the columns of $\beta$). We may present $\gamma$ in this way because of the Bruhat decomposition of $\Gamma$ into a partition of so-called Bruhat cells, each corresponding with a permutation $\sigma \in \sym_n$: 
$$ \Gamma = \bigsqcup_{\sigma \in \sym_n} BC_\sigma B, $$
where $B$ is the Borel subgroup of lower-triangular matrices inside of $\Gamma$.

Note that a lattice is primitive if and only if $\lambda_n=0$. We let a tuple $(I,\overline{r})$ record the elementary divisors of a given primitive lattice as follows: $I \coloneq \{ \iota \in [n-1] \ | \ \lambda_\iota > \lambda_{\iota + 1} \}$ and $ \overline{r} \in  \N^I_{>0}$ is defined by $r_\iota \coloneq \lambda_{\iota} - \lambda_{\iota+1}$. Here we are assuming that $n\geq 2$, which we will do from here on.

Consider a primitive lattice $\Lambda$ of type $(I,\overline{r})$ and in the Bruhat cell of $\sigma \in \sym_n$, represented by $D\beta C_\sigma$. Let $\Gamma_{(I,\overline{r})}$ denote the following (non-normal) subgroup of $\Gamma$:
$$ \Gamma_{(I,\overline{r})} \coloneq \{ g \in \Gamma \ | \ Dg \in \Gamma D \}\, . $$
One easily observes that $\Gamma_{(I, \overline{r})} = (D^{-1} \Gamma D) \, \cap \Gamma$ where $D^{-1}$ denotes the inverse of $D$ inside of $\GL_n(\mathbb{K})$ with $\mathbb{K}$ the field of fractions of $\O$. Clearly, $\beta$ is unique up to left multiplication with elements from $\Gamma_{(I,\overline{r})}$. Because of this, we can assume without loss of generality that $\beta$ is lower unitriangular with every lower-diagonal entry $\beta_{rc}$ (where $c<r$) determined modulo $\pi^{\sum_{\substack{ c \leq \iota < r}} r_\iota}$, where we are suppressing the notation $\iota \in I$, which will always be tacitly assumed.

\subsection{Polynomial congruences}\label{poly_cong_subsection}

 In order to determine when $\Lambda$ is a subalgebra, let $R(\overline{x})$ represent the matrix of linear forms in $n$ variables $x_1,\dots,x_n$ given by the multiplication structure of $L$: the entry $(i,j)$ of $R(\overline{x})$ is $l_{ij}(\overline{x}) \coloneq \sum_{q \in [n]} \lambda^q_{ij} x_q$, where the $\lambda^q_{ij}$ are the structure constants of $L$. The lattice $\Lambda$ is closed under multiplication (i.e. \  it is a subalgebra) if and only if for all $m \in [n]$, 
\begin{equation}\label{m_poly_matrix_cong}
    p^{\sum_{\iota < m} r_\iota} D \beta C_\sigma R(\beta^{-1}[m]) C^{-1}_\sigma \beta^t D \equiv 0 \bmod p^{\sum_{\iota \in I} r_\iota }.
\end{equation}

Any of the entries of the matrix on the left-hand side of this congruence are as follows. Define the following polynomial expressions in the entries of $\beta$:
\begin{equation}\label{klm_poly}
    f_{k,l}^m(\beta) \coloneq \sum_{i,j \in [n]} \beta_{ki} \beta_{lj} l_{ij}(\beta^{-1}[m]) 
\end{equation}
Then for the entry in position $(k,l)$ of \ref{m_poly_matrix_cong} we get
\begin{equation}\label{klm_poly=0}
    \pi^{ \sum_{\iota < m} r_\iota + \sum_{\iota \geq k} r_{\iota} + \sum_{\iota \geq l} r_\iota}  f_{k,l}^{m}(\beta) \equiv 0 \bmod \P^{\sum_{\iota \in I} r_\iota}
\end{equation}
Thus a primitive lattice, represented by $D\beta C_\sigma$, is a subalgebra if and only if for every $k,l,m \in [n]$, this last congruence \ref{klm_poly=0} is satisfied. For a non-primitive lattice $\pi^{r_n} \Lambda$, where $\Lambda$ is primitive and $r_n \in \N_{>0}$, closedness under multiplication is satisfied if and only if 
$$ \pi^{r_n} \cdot \pi^{ \sum_{\iota < m} r_\iota + \sum_{\iota \geq k} r_{\iota} + \sum_{\iota \geq l} r_\iota}  f_{k,l}^{m}(\beta) \equiv 0 \bmod \P^{\sum_{\iota \in I} r_\iota} $$
which is simply condition \ref{klm_poly=0}, weakened by a factor $\pi^{r_n}$. Note that this means that for any primitive $\Lambda$, if $r_n$ is large enough, these polynomial congruences can always be satisfied. This reflects the fact that for any homothety class, there are only finitely many members that are not subalgebras and that for any subalgebra, all homotheties are also subalgebras.

For a given type $(I,\overline{r})$ and permutation $\sigma \in \sym_n$, the primitive lattice is uniquely determined by (the lower-diagonal entries of) $\beta$. Thus, for a fixed type $(I,\overline{r})$ and permutation $\sigma$, we will often write $\Lambda_\beta$ to denote the primitive lattice associated with $\beta$. Whether or not $\Lambda_\beta$ is a subalgebra is then fully decided by the lower-diagonal entries $\beta_{rc}$.

\subsubsection*{Some notation}\label{splitting_notation_paragraph}

Recall from \ref{simple_pole_subsubsection} that $\tilde{a}^w_{q^i}(A)$ denotes the number of primitive lattices of index $q^i$ and weight $w$. Note that $\sum_{w=0}^i \tilde{a}^w_{q^i}(A) = \tilde{a}_{q^i}(A)$. Denote the number of primitive lattices with elementary divisor type $(I,\overline{r})$ by $\tilde{a}_{q^{I\cdot \overline{r}}}(A)$, and those of weight $w$ by $\tilde{a}^w_{q^{I\cdot \overline{r}}}(A)$. 

\subsection{Graded algebras and weighted homogeneity}

\begin{definition}\label{grading_definition}
    An $\O$-algebra $A$ is graded when it can be endowed with a decomposition into free $\O$-modules : $ A = \bigoplus_{j \in \N_{>0}} A_j$ such that $ A_k \cdot A_l \subseteq A_{k+l} $ for all $k,l \in \N_{>0}$.
\end{definition}

\noindent Let $d_j \coloneq  \dim (A_j)$. Note that since we are always assuming that $A$ is finite-dimensional, all but finitely many of the $A_j$ are trivial \textemdash \ i.e. $d_j=0$ for $j$ sufficiently large. Let $l \coloneq  \max \{j \ | \ d_j \neq 0\}$ denote the length of the grading. Note that if $A$ has a grading of length $l$, then $A$ is nilpotent of class at most $l$.

The main property of graded algebras that we will use is that they are \textit{homogeneous}:  the automorphism group of $A$ contains a subgroup isomorphic to $\O^\times$. Indeed, fix an ordered basis $( v_{11}, \dots, v_{1d_1}, \dots , v_{l1},\dots v_{ld_l} )$ for $A$ such that for every $j \in [l]$, the vectors $v_{j1},\dots,v_{jd_j}$ form a basis for $A_j$. There exists a non-decreasing set of weights \textit{weights} $w_1,...,w_n \in \N_{>0}$ such that the $\O$-linear map $\O^n \to  \O^n$ given by the matrix  $\T_\lambda \coloneq \text{diag} (\lambda^{w_1}, \dots, \lambda^{w_n}) $ is an isomorphism of $A$ for every $\lambda\in \O^\times$. The weights are defined as

$$ w_i \coloneq \min \left\{ k \in \N \ | \ \sum_{j = 1}^k \dim A_j \geq i \right\}. $$

\noindent Because the $\T_\lambda$-maps are automorphisms of $A$, $\T_{\lambda} x \cdot \T_\lambda y = \T_\lambda (x \cdot y)$ for all $x,y \in A$. Looking at the polynomials \ref{klm_poly} for a given $\sigma \in \sym_n$, this has the following important consequence: the support of $l_{ij}(\overline{x})$ consists of $x_q$ such that $w_{\sigma(i)} + w_{\sigma(j)} = w_{\sigma(q)}$. This reflects the fact that for any $x,y \in A$, the weight of their product must be the sum of their weights.

\begin{example}
    Any class-$2$ nilpotent $\O$-algebra $A$ is graded. Let $Z(A)$ denote $A$'s centre. Set $A_1 \coloneq A / Z(A)$ and $A_2 \coloneq Z(A)$. 
\end{example}

\begin{example}
    The free class-$c$ nilpotent Lie algebra on $d$ generators $F_{c,d}$ is graded. Indeed, let $x_1,\dots,x_d$ be a set of generators. Set $A_1 = \langle x_1,\dots,x_d \rangle$ as a free $d$-dimensional $\O$-module, and let $A_l$ for $l>1$ be inductively defined as $[A_{l-1}, A_1]$.  
\end{example}

\begin{example}
    Any nilpotent Lie algebra $L$ of dimension at most $6$ is graded (see \cite{quantization_nilpotent_Lie}, Remark 3.1.6 (2)) . There exists a $7$-dimensional (class-$6$) nilpotent Lie algebra (\cite{nilpotent_Lie_analysis}, Chapter 1 $\S$3, Example 2). However, its zeta function is unknown.
\end{example}

\noindent Our rationale for imposing that $A$ be graded is that this implies that the polynomials \ref{klm_poly} are weighted homogeneous. 

\begin{definition}\label{weighted_homo_def}
    A polynomial $f(x_1,...,x_m)$ over a ring $R$ is called \textit{weighted homogeneous} of degree $d \in \Z$ if there exist $w_1, ...,w_m \in \Z$, not all equal to $0$, such that for every $\lambda \in R^\times$,
    $$ f(\lambda^{w_1} x_1, \dots, \lambda^{w_m}x_m) = \lambda^d f(x_1,\dots, x_m). $$
    Equivalently, for every monomial $g$ in the support of $f$ we have
    $ \sum_{j=1}^m \text{mult}_{g}(x_j) \cdot w_j = d$,
    where $ \text{mult}_{g}(x_j) $ denotes the multiplicity of $x_j$ in $g$.
\end{definition}

\begin{proposition}
    If $A$ is a graded algebra, then the polynomials expressions \ref{klm_poly} are weighted homogeneous.
\end{proposition}

\begin{proof}
    The weight associated with entry $(r,c) \in [n]^2$ is $w_{\sigma(c)} - w_{\sigma(r)}$. Indeed, for $k,l,m \in [n]$, we have the polynomial expressions
    $$ f^m_{k,l}(\beta) \coloneq \sum_{i,j \in [n]} \beta_{ki} \beta_{lj} l_{ij} (\beta^{-1}[m]), $$
    and under the map $\beta_{r,c} \longmapsto  \lambda^{w_{\sigma(r)} - w_{\sigma(c)}} \beta_{r,c}$ this transforms into
    \begin{align*}
        &\sum_{i,j \in [n]} \lambda^{w_{\sigma(i)} - w_{\sigma(k)}}\beta_{ki} \lambda^{w_{\sigma(j)} - w_{\sigma(l)}} \beta_{lj}  (\lambda^{-1})^{w_{\sigma(m)} - (w_{\sigma(i)} + w_{\sigma(j)})} l_{ij} (\beta^{-1}[m])\\ &= \lambda^{w_{\sigma(m)} - w_{\sigma(k)} - w_{\sigma(l)}} f^m_{k,l}(\beta).
    \end{align*}
\end{proof}

\subsection{Proof of Theorem \ref{grading_residue_theorem}}\label{any_type_subsection}

In this subsection we fix an elementary divisor type $(I,\overline{r})$, define $R \coloneq \max_{\iota \in I} \{ r_\iota \}$, and fix a $\iota^* \in I$ such that $R = r_{\iota^*}$. We will prove the following.
\begin{theorem}\label{any_type_theorem}
    For an $n$-dimensional graded  $\O$-algebra $A$, we have for all $w\in \N$,
    $$ v_q(\tilde{a}^w_{q^{I\cdot \overline{r}}}(A)) \geq \left\lfloor \frac{R+w}{2(n-1)} \right \rfloor.  $$
\end{theorem}

\subsubsection{Multiplicative action: dilations}\label{mult_subsection}

We let $\O^\times$ act via the group of dilation automorphisms induced by the grading on $A$. That is, for $\lambda \in \O^\times$, we have
$$ \Gamma M \longmapsto \Gamma M \cdot \T_\lambda $$
Given a $\sigma \in \mathcal{S}_n$, for the lattices $\Lambda_\beta$ associated with the Bruhat cell of $\sigma$, this action corresponds to
\begin{equation}\label{mult_action}
    \Lambda_\beta \longmapsto \Lambda_{\beta'}, \ \text{ where } \ \beta' \coloneq (\T_\lambda^\sigma)^{-1} \beta\,{\T_\lambda^\sigma}
\end{equation}
since $\T_\lambda^\sigma \in \Gamma_{(I,\overline{r})}$. Consider the orbit $\O^\times \cdot \Gamma M$. Any given $\lambda \in \O^\times$ stabalizes $\Gamma M$ if and only if $\beta \T_\lambda^\sigma \beta^{-1} \in \Gamma_{(I,\overline{r})} $. Thus
\begin{equation}\label{stab_lambda}
    \text{Stab}_{\O^\times} (\Lambda_{\beta}) = \{ \lambda \in \O^{\times} \, | \, \forall (r,c) \in [n]^2 : \ \sum_{k = c}^r \lambda^{w_{\sigma(k)}} \beta_{rk} (\beta^{-1})_{kc} \equiv 0 \bmod \P^{\sum_{c\leq \iota < r} r_\iota}  \}.
\end{equation}

\noindent Denote the orbit of $\Lambda_\beta$ under the action by $\O^\times$ as $\text{orbit}(\beta)$. We will be interested in $v_q(\# \text{orbit}(\beta))$. In particular, if $\stab_{\O^\times}(\beta) \leq \O^\times+\pi^{v} \O$ for some positive integer $v$, then by the Orbit-Stabilizer Theorem $v_q(\#\text{orbit}(\beta)) \geq v-1$. Indeed,
\begin{align*}
    v_q(\#\text{orbit}(\beta)) &= v_p([\O^\times : \stab_{\O^\times}(\beta)]) \\
    &= v_q([\O^\times : \O^\times + \pi^v \O]) + v_{q}([\O^\times: \stab_{\O^\times}(\beta)]) \\
    &\geq v_q([\O^\times : \O^\times + p^v \O]) = v-1 \, .
\end{align*}

\begin{lemma}\label{beta_induction}
For any $r,c \in [n]^2$ with $r>c$, we have
    \begin{align*}
        (\beta \T^\sigma_\lambda \beta^{-1})_{rc} &=  \beta_{rc}(\lambda^{w_{\sigma(c)}} - \lambda^{w_{\sigma(r)}}) +   \sum_{ c < k < r }  (\beta \T^\sigma_\lambda \beta^{-1})_{rk} \beta_{kc}\\
        &=  (\beta^{-1})_{rc}(\lambda^{w_{\sigma(c)}} - \lambda^{w_{\sigma(r)}}) +  \sum_{ c < k < r } (\beta^{-1})_{rk}  (\beta \T^\sigma_\lambda \beta^{-1})_{kc} \, .
    \end{align*}
\end{lemma}

\begin{proof} Proof of the first equality:
    \begin{align*}
         \beta_{rc}(\lambda^{w_{\sigma(c)}} - \lambda^{w_{\sigma(r)}}) +   \sum_{ c < k < r }  (\beta \T^\sigma_\lambda \beta^{-1})_{rk} \beta_{kc} &=  \sum_{ c \leq k \leq r }  (\beta \T^\sigma_\lambda \beta^{-1})_{rk} \beta_{kc}  \\ 
         &= \sum_{ c \leq k \leq r }  \sum_{l=k}^r \lambda^{w_{\sigma(l)}} \beta_{rl} (\beta^{-1})_{lk} \beta_{kc} \\
         &= \sum_{l=c}^r \lambda^{w_{\sigma(l)}} \beta_{rl} \sum_{k=c}^l (\beta^{-1})_{lk} \beta_{kc} \\
         &= \sum_{l=c}^r \lambda^{w_{\sigma(l)}} \beta_{rl} (\beta^{-1})_{lc} = (\beta \T^\sigma_{\lambda} \beta^{-1})_{rc}\ .
    \end{align*}
    The proof of the second equality is analogous.
\end{proof}

\begin{lemma}\label{val_beta_influenced}
    Choose $j \leq \iota^*$ such that $w_{\sigma(j)} < \min_{r > \iota^*} \{w_{\sigma(r)}\}$. Then for all $l > \iota^*$, we have 
    $$ \nu(\beta_{lj}) \geq \sum_{\iota^* \leq \iota < l} r_{\iota} - (v_q(\#\text{orbit}(\beta))+1)(l-\iota^*). $$
\end{lemma}
\noindent From this lemma we deduce 
    $$ \nu(\beta_{lj}) \geq \sum_{\iota^* \leq \iota < l} r_{\iota} - (v_q(\#\text{orbit}(\beta))+1)(n-\iota^*)  \geq \sum_{\iota^* \leq \iota < l} r_{\iota} - (v_q(\#\text{orbit}(\beta))+1)(n-1). $$
Thus, inspired by this, define the invariant
$ \varepsilon(\beta) \coloneq (v_q(\#\text{orbit}(\beta))-1)(n-1). $
 
\begin{proof}[Proof of Lemma \ref{val_beta_influenced}]
    Let $j$ be as in the statement of the Lemma. We proceed by induction on $l$. For the base step, let $l = \iota^*+1$. Then for any $\lambda \in \O^\times$, by Lemma \ref{beta_induction} (first equation):
    \begin{equation}\label{beta_induction_eq}
        (\beta \T_\lambda\beta^{-1})_{lj} =  \beta_{lj}(\lambda^{w_{\sigma(j)}} - \lambda^{w_{\sigma(l)}}) +  \sum_{k=j+1}^{l-1} \beta_{kj}  (\beta \T_\lambda \beta^{-1})_{lk}\, .
    \end{equation}
    If $\lambda \in \stab_{\O^\times}(\beta)$, then  $(\beta \T_\lambda \beta^{-1})_{rc} \equiv 0 \bmod \P^{\sum_{c \leq  \iota <r} r_\iota}$ for all $c<r$. In particular, in Equation \ref{beta_induction_eq} we see that every $ (\beta \T_\lambda \beta^{-1})_{lk} \equiv 0 \bmod \P^R$ for $j<k<l$, since $\sum_{l-1 \leq \iota < l} r_\iota = R$. Thus, on the one hand we know that
    
    $$ (\beta \T_\lambda\beta^{-1})_{lj} \equiv 0 \bmod \P^{\sum_{j \leq \iota < l} r_\iota}, $$
    and on the other hand we have 
    $$  (\beta \T_\lambda\beta^{-1})_{lj} \equiv \beta_{lj}(\lambda^{w_{\sigma(j)}} - \lambda^{w_{\sigma(l)}}) \bmod \P^{R}. $$
    Therefore 
    $$  \beta_{lj}(\lambda^{w_{\sigma(j)}} - \lambda^{w_{\sigma(l)}}) \equiv 0 \bmod \P^{R}, $$
    since $R \leq \sum_{j\leq \iota < l} r_\iota$. Because $w_{\sigma(j)} < w_{\sigma(l)}$, this congruence is satisfied if and only if $\nu(\lambda-1) \geq R - \nu(\beta_{lj})$. That is, $\stab_{\O^\times}(\beta) \leq \O^\times+\pi^{R-\nu(\beta_{lj})}\O$. This implies that $v_q(\#\text{orbit}(\beta)) \geq R - \nu(\beta_{lj}) -1$, and rearranging, we get $\nu(\beta_{lj}) \geq R - v_q(\#\text{orbit}(\beta))-1$, which proves the base step.
    
    For the induction step, assume that $l \geq \iota^*+2$ and that the lemma holds for all $l' < l$: that is, for $\iota^*<l'<l$, we have
    $$ \nu(\beta_{l'j}) \geq \sum_{\iota^* \leq \iota < l'} r_\iota - \varepsilon(\beta). $$
    Observe again that
    \begin{align*}
        (\beta \T_\lambda\beta^{-1})_{lj} &= \beta_{lj}(\lambda^{w_{\sigma(j)}} - \lambda^{w_{\sigma(l)}}) + \sum_{j<k<l} \beta_{kj} (\beta \T_\lambda \beta^{-1})_{lk} \\
        &= \beta_{lj}(\lambda^{w_{\sigma(j)}} - \lambda^{w_{\sigma(l)}}) + \sum_{j<k\leq \iota^{*}} \beta_{kj} (\beta \T_\lambda \beta^{-1})_{lk}  + \sum_{\iota^{*}<k<l} \beta_{kj} (\beta \T_\lambda \beta^{-1})_{lk} \, .
    \end{align*}
    For $j<k\leq \iota^{*}$ we know that $\nu((\beta \T_\lambda \beta^{-1})_{lk}) \geq \sum_{k\leq \iota < l} r_\iota \geq \sum_{\iota^* \leq \iota < l} r_\iota$. For $\iota^{*}<k<l$ we use the induction hypothesis to say that 
    \begin{align*}
        \nu(\beta_{kj} (\beta \T_\lambda \beta^{-1})_{lk}) &\geq \sum_{\iota^* \leq \iota < k} r_\iota - (v_q(\# \text{orbit}(\beta) )+1)( k-\iota^* ) + \sum_{k\leq \iota < l} r_\iota \\ 
        &= \sum_{\iota^* \leq \iota < l} r_\iota - (v_q(\# \text{orbit}(\beta) )+1)( k-\iota^* ) \\
        &\geq \sum_{\iota^* \leq \iota < l} r_\iota - (v_q(\# \text{orbit}(\beta) )+1)( l-1-\iota^* ) \, .
    \end{align*}
    
    \noindent Therefore
    \begin{equation*}
        (\beta \T_\lambda\beta^{-1})_{lj} \equiv \beta_{lj}(\lambda^{w_{\sigma(j)}} - \lambda^{w_{\sigma(l)}}) \bmod \P^{\sum_{\iota^* \leq \iota < l} r_\iota - (v_q(\# \text{orbit}(\beta) )+1)( l-1-\iota^* )}  .
    \end{equation*}
    Thus, as in the base case, we have $$\stab_{\O^\times}(\beta) \leq \O^\times + \pi^{\sum_{\iota^* \leq \iota < l} r_\iota - (v_q(\# \text{orbit}(\beta) )+1)( l-1-\iota^* ) - \nu(\beta_{lj})} \O$$ which means that 
    \begin{align*}
        v_q(\# \text{orbit}(\beta)) &\geq  \sum_{\iota^* \leq \iota < l} r_\iota - (v_q(\# \text{orbit}(\beta) )+1)( l-1-\iota^* ) - \nu(\beta_{lj}) -1
    \end{align*}
    and thus
    \begin{align*}
        \nu(\beta_{lj}) &\geq  \sum_{\iota^* \leq \iota < l} r_\iota - (v_q(\# \text{orbit}(\beta) )+1)( l-1-\iota^* ) - v_q(\# \text{orbit}(\beta)) -1 \\
        &=  \sum_{\iota^* \leq \iota < l} r_\iota - (v_q(\# \text{orbit}(\beta) )+1)( l-\iota^* ),
    \end{align*}
    which is what had to be proven.
\end{proof}

\begin{lemma}\label{val_beta-1_influenced}
    Choose $q > \iota^*$ such that $w_{\sigma(q)} > \max_{c \leq \iota^*} \{w_{\sigma(c)}\}$. Then for all $m \leq \iota^*$ we have
    $$ \nu((\beta^{-1})_{qm}) \geq \sum_{m \leq \iota \leq \iota^*} r_{\iota} - ( v_q(\#\text{orbit}(\beta))+1 )(\iota^* - m) \geq \sum_{m \leq \iota \leq \iota^*} r_{\iota} - \varepsilon(\beta).  $$
\end{lemma}

\begin{proof}
    Analogous to \ref{val_beta_influenced}.
\end{proof}

\subsubsection{Additive action: transvections}\label{add_subsection}

We choose $r^* > \iota^*$ and $c^* \leq \iota^*$ such that the corresponding weight is minimized (resp. maximised):
$$ w_{r^*} = \min_{r > \iota^*} \left\{ w_{\sigma(r)} \right\}, \quad w_{c^*} = \max_{c \leq \iota^*} \left\{ w_{\sigma(c)} \right\}.  $$
Let $M_{r^*c^*}(\mu)$ for $\mu \in \O$ denote the following elementary matrix:
$$ (M_{r^*c^*}(\mu))_{ij} \coloneq \begin{cases} \delta_{ij} \quad &\text{ for } \ (i,j) \neq (r^*,c^*)  \\ \mu \quad &\text{ for } \ (i,j) = (r^*,c^*)\ . \end{cases}$$
We define an additive action by $\O$ via right multiplication by $M_{r^*c^*}(\mu)$:
$$ \Gamma M = \Gamma D_{(I,\overline{r})} \beta C_{\sigma} \longmapsto \Gamma D_{(I,\overline{r})} (\beta \cdot M_{r^*c^*}(\mu)) C_{\sigma}. $$
Thus, one the level of lattices, this corresponds to
\begin{equation}\label{add_action}
    \Lambda_{\beta} \longmapsto \Lambda_{\beta M_{r^* c^*}(\mu)}. 
\end{equation}

\noindent Again, to study the orbit size of $\beta$ under this action, we look at the stabilizer:
$$ \stab_{\O} (\beta) = \left\{ \mu \in \O  \, \middle| \, (M_{r^*c^*}(\mu))^{\beta^{-1}} \in \Gamma_{(I,\overline{r})} \right\}. $$
Because
$$ (\beta \Delta_{r^*c^*}(\mu) \beta^{-1})_{ij} = \mu \cdot \beta_{ir^*} (\beta^{-1})_{c^*j}, $$
this means that 
\begin{equation*}\label{stab_mu}
    \stab_{\O}(\beta) = \left\{ \mu \in \O \, \middle| \,  \mu \cdot \beta_{ir^*} (\beta^{-1})_{c^*j} \equiv 0 \bmod \P^{\sum_{j \leq \iota <i} r_\iota} \ \text{ for all } \ (i,j)\in [n]^2 \right\}.
\end{equation*}
Thus $\stab_{\O}(\beta) = \pi^\delta \O$ for some $\delta \in \N$. Define $\delta(\beta) \coloneq v_q( \#(\O \cdot \Lambda_\beta))$.
\begin{lemma}\label{val_beta_added}
    $$ \nu(\beta_{ir^*}) \geq \sum_{ c^* \leq \iota < i } r_\iota - \delta(\beta) \ \text{ for all } i \in [n]$$
\end{lemma}

\begin{proof}
    For $j = c^{*}$, the stabilizer condition in \ref{stab_mu} reduces to $ \mu \cdot \beta_{ir^*} \equiv 0 \bmod \P^{\sum_{ c^* \leq \iota < i } r_\iota } $. Thus, if $\nu(\beta_{ir^*}) < \sum_{ c^* \leq \iota < i } r_\iota - \delta'$ for some $\delta' \in \N$, then $\mu = \pi^{\delta'}$ is not a solution of this congruence, and thus not an element of $\stab_{\O} (\beta)$.
\end{proof}

\begin{lemma}\label{val_beta-1_added}
    $$ \nu((\beta^{-1})_{c^*j}) \geq \sum_{j \leq \iota < r^*} r_\iota - \delta(\beta) \ \text{ for all } \ j \in [n] $$
\end{lemma}

\begin{proof}
    Analogous to the proof of Lemma \ref{val_beta_added}.
\end{proof}

\noindent Recall the polynomial congruences \ref{klm_poly=0} that $\beta$ must satisfy for $\Lambda_\beta$ to be a subalgebra:

$$ f_{k,l}^{m}(\beta) \, \pi^{\sum_{\iota < m} r_\iota + \sum_{\iota \geq k} r_\iota + \sum_{\iota \geq l} r_\iota} \equiv 0 \bmod \P^{\sum_{\iota \in I} r_\iota}, $$
where
$$ f_{k,l}^{m}(\beta) \coloneq \hspace{-0.1cm} \sum_{i,j \in [n]} \beta_{ki} \beta_{lj} l_{ij}(\beta^{-1}[m]), \ \text{ and } \ l_{ij}(\beta^{-1}[m]) = \hspace{-0.9cm} \sum_{ \substack{q \in [n] \\ w_{\sigma(i)} + w_{\sigma(j)} = w_{\sigma(q)} }} \hspace{-0.7cm} \lambda^q_{i,j} (\beta^{-1})_{qm}. $$

\noindent Using this, we can prove the following.

\begin{theorem}\label{any_type_weight_additive_theorem}
    If $\beta$ represents a primitive lattice $\Lambda_\beta$ of weight $w$, then the primitive lattice $\Lambda_{\beta M_{r^*c^*}(\mu)}$ has weight $w$ whenever
    $$ \nu(\mu) \geq \max\left\{ \delta(\beta) - \frac{R+w}{2}, \delta(\beta) - (R+w-\varepsilon(\beta)) \right\}.$$
\end{theorem}

\begin{proof} 

First, we prove this for weight $w=0$. We want to show that if $\beta$ represents a subalgebra, then $\beta M_{r^*c^*}(\mu)$ represents a subalgebra whenever
    $$\nu(\mu) \geq \max \left\{ \delta(\beta) - R/2 , \delta(\beta) - (R-\varepsilon(\beta)) \right\}.$$

\noindent The action $\beta \mapsto \beta M_{r^*c^*} (\mu)$ changes the polynomial expressions $f^m_{k,l}(\beta)$ as follows:

\begin{align*}
    f_{k,l}^m (\beta M_{r^*c^*}(\mu) ) = f_{k,l}^m(\beta)
    &\ + \mu \cdot \sum_{j\in [n]} \beta_{kr^*} \beta_{lj} l_{c^*j}(\beta^{-1}[m]) \\ 
    &\ + \mu \cdot \sum_{i\in [n]} \beta_{ki} \beta_{lr^*} l_{ic^*}(\beta^{-1}[m]) \\
    &\ + \mu \cdot \hspace{-0.35cm} \hspace{-0.7cm} \sum_{\substack{i,j \in [n]^2 \\ w_{\sigma(i)} + w_{\sigma(j)} = w_{\sigma(r^*)}}} \hspace{-0.8cm} \beta_{ki} \beta_{lj} \lambda_{ij}^{r^*} (\beta^{-1})_{c^* m} \  + \mathcal{O}(\mu^2).
\end{align*}

\noindent Consider the linear terms in $\mu$ one by one. In the first sum, every summand contains, for $j\in [n]$, the factors $\beta_{kr^*}$ and $l_{c^*j}(\beta^{-1}[m])$. By Lemma \ref{val_beta_added}, we know that
$$ \nu(\beta_{kr^*}) \geq \sum_{c^* \leq \iota < k} r_\iota - \delta(\beta) \geq \sum_{\iota^* \leq \iota < k } r_{\iota} - \delta(\beta)$$
Furthermore, by Lemma \ref{val_beta-1_influenced},

$$ \nu(l_{c^*j}(\beta^{-1}[m]) ) \geq  \sum_{m \leq \iota \leq \iota^*} r_{\iota} - \varepsilon(\beta) =  R+\sum_{m \leq \iota < \iota^*} r_{\iota} - \varepsilon(\beta).$$
Therefore, because $\nu(\mu) \geq \delta(\beta) - (R-\varepsilon(\beta)) $, we have

\begin{equation}\label{mu1}
    \pi^{\sum_{k \leq \iota} r_\iota + \sum_{\iota < m} r_{\iota}} \mu \beta_{kr^*} l_{c^*j} (\beta^{-1}[m]) \equiv 0  \bmod \P^{\sum_{\iota \in I} r_\iota} .
\end{equation}

\noindent Similarly, for the second sum in the linear part, we see that for every $i \in [n]$, in the summand corresponding to $i$, there are the factors $\beta_{lr^*}$ and $l_{ic^*}(\beta^{-1}[m])$. Given that by Lemma \ref{val_beta_added}
$$ \nu(\beta_{lr^*}) \geq \sum_{\iota^* \leq \iota < l} r_\iota - \delta(\beta) $$
and that by Lemma \ref{val_beta-1_influenced}
$$ \nu \left(l_{ic^*}(\beta^{-1}[m])\right) \geq \sum_{m \leq \iota \leq \iota^*} r_{\iota} - \varepsilon(\beta)= R + \sum_{m \leq \iota < \iota^*} r_{\iota} - \varepsilon(\beta)\, . $$
Therefore
\begin{equation}\label{mu2}
    \pi^{\sum_{l \geq \iota} r_\iota + \sum_{\iota < m} r_{\iota}} \mu \beta_{lr^*} l_{ic^*} (\beta^{-1}[m]) \equiv 0  \bmod \P^{\sum_{\iota \in I} r_\iota}.
\end{equation}

\noindent Lastly, for any $(i,j) \in [n]^2$ such that $w_{\sigma(i)} + w_{\sigma(j)} = w_{\sigma(r^*)}$ we know that by Lemma \ref{val_beta_influenced}
$$ \nu(\beta_{ki}) \geq R+ \sum_{\iota^* < \iota < k} r_{\iota} - \varepsilon(\beta), \qquad \nu(\beta_{lj}) \geq R+\sum_{\iota^* < \iota < l} r_{\iota} - \varepsilon(\beta)  $$
and by Lemma \ref{val_beta-1_added} that
$$ \nu((\beta^{-1})_{c^*m}) \geq \sum_{m \leq \iota < r^*} \! r_\iota - \delta(\beta). $$
Therefore
\begin{equation}\label{mu3}
     \pi^{ \sum_{ k\leq  \iota} r_\iota + \sum_{\iota < m} r_\iota } \mu \beta_{ki}\beta_{lj} (\beta^{-1})_{c^*m} \equiv 0 \bmod \P^{\sum_{\iota \in I} r_{\iota}}.
\end{equation}

\noindent Combining \ref{mu1}, \ref{mu2}, and \ref{mu3}, we may conclude that for all $(k,l,m) \in [n]^3$
$$ f^m_{k,l}(\beta M_{r^*c^*}(\mu)) \, \pi^{\sum_{m>\iota} r_\iota + \sum_{k\leq \iota} r_\iota + \sum_{l \leq \iota} r_\iota} \equiv f^m_{k,l}(\beta) \, \pi^{\sum_{m>\iota} r_\iota + \sum_{k\leq \iota} r_\iota + \sum_{l \leq \iota} r_\iota} \bmod \P^{\sum_{\iota \in I} r_{\iota}}. $$
Note that one can easily check that all $\mu^2$-terms vanish, using the assumption that $\nu(\mu) \geq \delta(\beta) - R/2$. Therefore, if $\beta$ represents a subalgebra, so does $\beta M_{r^*c^*}(\mu)$. 

For positive weight $w$, the proof is analogous. This time, we consider the congruences
    $$ \pi^w f^m_{k,l}(\beta M_{r^*c^*}(\mu)) \, \pi^{\sum_{m>\iota} r_\iota + \sum_{k\leq \iota} r_\iota + \sum_{l \leq \iota} r_\iota} \equiv 0 \bmod \P^{\sum_{\iota} r_\iota} $$
    where one easily checks that the $\mu^2$-terms vanish using $\nu(\mu) \geq  \delta(\beta) - \frac{R+w}{2}$. Using the same reasoning as for $w=0$, we can show that all $\mu$-linear terms have valuation at least 
    $$ w + \sum_{\iota \in I} r_\iota - (\delta(\beta) - (R - \varepsilon(\delta))), $$
    which means that they vanish for $ \nu(\mu) \geq \delta(\beta) - (R+w-\varepsilon(\delta)) $.
\end{proof}

\subsubsection{Proof of Theorem \ref{grading_residue_theorem} }

Now we combine the multiplicative and additive actions.

\begin{proof}[Proof of Theorem \ref{any_type_theorem}]

Let $\beta$ represent any lattice $\Lambda_\beta$ of type $(I,\overline{r})$, weight $w$, and associated with $\sigma \in \sym_{n}$. Then either $\varepsilon(\beta) > \frac{R+w}{2}$ or $\varepsilon(\beta) \leq \frac{R+w}{2}$. In the former case, since $\varepsilon(\beta) = (v_q(\#\text{orbit}(\beta))+1)(n-1)$, the orbit the orbit of $\Lambda_\beta$ under the multiplicative action by $\O^\times$ has a size of valuation $v_q(\#\text{orbit} (\beta)) > \frac{R+w}{2(n-1)} -1$. In the latter case, $\Lambda_\beta$ lies inside an additive orbit whose size has valuation at least $(R+w)/2$ by Theorem \ref{any_type_weight_additive_theorem}.
\end{proof}

\begin{corollary}\label{absolute_convergence_iw}
    Let $A$ be an $n$-dimensional graded $\O$-algebra. Then for any $(i,w) \in \N^2$,
    $$ v_q\left(\tilde{a}^w_{q^i}(A)\right) \geq \left\lfloor \frac{i}{n(n-1)^2} \right\rfloor. $$
\end{corollary}

\begin{proof}
    If a lattice has index $q^i$ and type $(I,\overline{r})$, then by definition $i = \sum_{\iota\in I} \iota \cdot r_\iota$. Clearly, this means that $i \leq \max_{\iota \in I} \{ r_\iota \} \cdot \binom{n}{2}$. Therefore, by Theorem \ref{any_type_theorem},
    $$ v_q\left(\tilde{a}^w_{q^i}(A)\right) > \frac{i}{2(n-1)\binom{n}{2}} -1 = \frac{i}{n(n-1)^2} -1. $$
\end{proof}

\begin{proof}[Proof of Theorem \ref{grading_residue_theorem}]
    Recall the definition $c_k(A) \coloneq \sum_{w=0}^{\lfloor k/({n+1}) \rfloor} \tilde{a}^w_{q^{k-n\cdot w}}(A)$.
    Corollary \ref{absolute_convergence_iw} implies that $ v_q(c_k(L)) \to \infty \ \text{ as } k \to \infty$. As mentioned in \ref{simple_pole_subsubsection}, this means that $\zeta_{A}(t)$'s pole at $t=1$ is simple. Moreover, by Fubini's Theorem \ref{fubini}, Equation \ref{i_vs_k} holds:
    $$ \sum_{k \in \N} c_k(A) = \sum_{i=0}^{\infty} \sum_{w=0}^i \tilde{a}^w_{q^i}(A). $$
    By Equation \ref{sumc=nres}, we know that the left hand side is $n \cdot \text{res}_{t=1} \zeta_{A}(t) = [(1-t^n)\zeta_{A}(t)]_{t=1}$. The right hand side is
    $$ \sum_{i=0}^{\infty} \tilde{a}_{q^i}(A) = \sum_{i=0}^{\infty} \tilde{a}_{q^i}(\O^n) = n\binom{\infty}{n-1}_q. $$
    Therefore
    $$ \text{res}_{t=1} \zeta_{A}(t) = \binom{\infty}{n-1}_q $$
    as needed to be proven.
\end{proof}

\noindent We repeat the significance of this result in terms of the $p$-asymptotics of $A$.

\begin{corollary}
    Let $A$ be an $n$-dimensional graded $\O$-algebra. Then 
    $$ a_{q^i}(A) \underset{p}{\longrightarrow} \frac{1}{(1-q)(1-q^2) \dots (1-q^{n-1})} \ \text{ as } \ i \to \infty . $$
\end{corollary}

\section{Smallest real pole of Igusa zeta functions}
\addtocontents{toc}{\protect\setcounter{tocdepth}{1}}

\noindent The methods developed in this paper fit within the broader framework of Igusa zeta functions. Given a polynomial map $f: \O^n\to \O^m: \overline{x} \mapsto (f_1(\overline{x}),\dots,f_m(\overline{x}))$, its Igusa zeta function is a complex function defined by a $p$-adic integral as follows \cite{igusa_intro}: 
$$ I_f(s) \coloneq \int_{\O^n} \| f(\overline{x}) \|^s d\mu  $$
where $\|y\| \coloneq  \max_{j=1}^m \{ q^{-\nu(y_j)} \}$ for $y\in \O^m$, and $\mu$ denotes the Haar measure, conventionally normalized to set $\mu(\O^n) = 1$. Igusa zeta functions can be seen as power series in $t\coloneq q^{-s}$ which count solutions to polynomial congruences. Indeed, let $M_{i}(f) \coloneq \#\{ \overline{a} \in \O^n \ | \ f(\overline{a}) \equiv 0 \bmod \P^i \}$. One easily checks that
\begin{equation}\label{Igusa_Poincare}
    \frac{1-q^{-s}I_f(s)}{1-q^{-s}} = \sum_{i = 0}^{\infty} M_i(f) \, q^{-(s+n)i} . 
\end{equation}
Igusa \cite{igusa_I} proved that $I_f(s)$ is rational in $t$ for $m=1$, using Hironaka's theorem on resolution of singularities \cite{hironaka_resolution}. 
Denef \cite{denef_84} gave a different proof, which relied on a model-theoretic result by Macintyre \cite{macintyre_qe} on quantifier elimination in $\Q_p$. These results in turn allowed Grunewald, Segal, and Smith to prove the rationality of local subalgebra and ideal zeta functions \cite{gss}. They did this by showing that these zeta functions are closely related to certain Igusa zeta functions for a well-chosen set of polynomials.

Further results by Denef and van den Dries \cite{denef_vdd} gave du Sautoy \cite{duSau_uniformity} the tools to prove the denominator form \ref{Q(T)=prod}, mentioned in the introduction. They establish the following. Let $I_f(s) = \frac{P(t)}{Q(t)}$ be reduced for two integer polynomials $P(T), Q(T) \in \Z[T]$. Then there exist distinct pairs of integers $(a_1, b_1), \dots, (a_e, b_e)$, and nonnegative integers $d_1, \dots, d_e$, such that $Q(T)$ divides $\prod_{j=1}^e (q^{a_j} - T^{b_j})^{d_j+1}$. Note the similarity with expression \ref{Q(T)=prod}. The pairs $(a_j,b_j)$ are the numerical data inherited from the resolution of singularities. Thus, the real pole spectrum of $I_f(s)$ is contained in $\{-\frac{a_1}{b_1}, \dots , -\frac{a_e}{b_e}\}$. By Equation \ref{Igusa_Poincare}, this means that there exist quasi-polynomials $R_1(x),\dots, R_e(x)$ over $\Q$, with $R_j$ of degree $d_j$ and period $b_j$, such that,
$$ M_i(f) = \sum_{j=1}^e R_j(i) \left(q^{n-\frac{a_j}{b_j}}\right)^i. $$
This is analogous to Equation \ref{a=sum_exp}. As we remarked with Equation \ref{a_to_smallest_pole_eq}, the exponential term for which $n - \frac{a_j}{b_j}$ is minimal dominates the $p$-adic asymptotics. In particular,

\begin{equation}\label{poleinf=liminf}
    n-\max_{j}\left(\frac{a_j}{b_j} \right) = \liminf_{i\to \infty} \left( \frac{v_q(M_i(f))}{i}\right).
\end{equation}

\noindent Using homothety classes and weights, Voll \cite{voll_functional} established another connection between subalgebra and Igusa zeta functions, different from the one by Grunewald, Segal, and Smith \cite{gss}. In the present paper, we favoured Voll's approach, which we outlined in Subsection \ref{poly_cong_subsection}. By showing that counting subalgebras is equivalent to counting primitive sublattices by weight, which in turn amounts to enumerating solutions to the polynomial congruences \ref{klm_poly=0}, Voll proved that, by Equation \ref{Igusa_Poincare}, subalgebra zeta functions can, with some modifications, be identified with the Igusa zeta functions corresponding to those polynomial expressions.

Much research has been devoted to the pole spectrum of Igusa zeta functions \cite{denef_vdd, meuser_poles_curves, veys_poles_monodromy_93, zunifa_poles_alg_sets, veys_zuniga_analytic}. In practice, however, computing Igusa zeta functions quickly becomes impracticable, especially if $f$ is a highly singular system of polynomials. This is why structural properties of the relevant polynomials, such as those which allowed us to prove Theorem \ref{any_type_theorem} (e.g. weighted homogeneity), are key to making general statements about the poles of the associated zeta functions. With the following result, we attempt to generalize the ideas of this paper to the context of Igusa zeta functions. Let $J_f(a)$ denote $f$'s Jacobian at a point $a \in \O^n$ and consider its elementary divisors $\delta_j(J_f(a))$ for $j\in [n]$.

\begin{theorem}[Reverse Hensel]\label{reverse_hensel}
    Given a system of $m$ polynomials $f_1, ..., f_m \in \O[X_1, ..., X_n]$. Let $\lambda_1,\dots, \lambda_n \in [0,1]$ be such that for all $a \in \O^n$ we have
    $$ 2\cdot \nu(\delta_j(J_f(a))) \geq \lambda_j \cdot \nu(f(a)). $$
    Then for any pole $s_0$ of $I_f(s)$,
    $$ \textup{Re}(s_0) \geq -n  + \sum_{j=1}^n \frac{\lambda_j}{2} \, . $$
\end{theorem}

\begin{proof}
    If $m <n$, set $f_{m+1},\dots,f_n \coloneq  0$. Thus, without loss of generality, we may assume $m \geq n$. Let $a \in \O^n$ be a solution of $f(x)$ modulo $\pi^{2i}$ for some $i \in \N$: $ f(a) \equiv 0 \bmod \P^{2i} $. Then for any $b \in \O^n$, we have by Taylor expansion
    $$ f(a + \pi^i b)  = f(a) + \pi^iJ_f(a)b + \mathcal{O}(\pi^{2i}) \equiv f(a) + \pi^iJ_f(a)b \bmod \P^{2i}.$$
    Thus, in order for $b$ to be a solution to $f(a + \pi^i b)  \equiv 0 \bmod \P^{2i}$, it must solve the system of linear congruences given by
    $$ J_f(a) b \equiv 0 \bmod \P^i. $$
    It is a well-known fact about such systems that the number of solutions is determined by the elementary divisors of $J_f(a)$. More precisely, if $\nu_1,\dots,\nu_n$ are the valuations of the elementary divisors, then the number of solutions is $q^{\sum_j\min\{\nu_j,i\}}$. By assumption $2\nu_j \geq \lambda_j \cdot 2i$, and thus $\min\{\nu_j, i\} \geq \min\{\lambda_j \cdot i, i\} = \lambda_j \cdot i $. Hence, for every solution $a$ modulo $\pi^{2i}$, we have proven that the ball $\left((a+\pi^i\O) / (\pi^{2i} \O)\right)^n$ contains a number of solutions with $q$-adic valuation at least $i\cdot\sum_{j=1}^n \lambda_j$.

    Thus, $v_q(M_{2i}(f)) \geq i\cdot\sum_{j=1}^n \lambda_j$ for all $i\in \N$, and similarly $v_q(M_{2i+1}(f)) \geq i\cdot\sum_{j=1}^n \lambda_j$. By Formula \ref{poleinf=liminf}, this proves the theorem.
\end{proof}

\begin{remark}
    A similar theorem was proven by Segers \cite{segers_lower_bound} for a single polynomial in $n\geq2$ variables. He states that all poles have real part at least $-\frac{n}{2}$, which is optimal as there exist polynomials $f$ for which this is a pole of $I_f(s)$. One easily checks that for $m=1$ and $n\geq 2$, Theorem \ref{reverse_hensel} implies that all poles of the associated Igusa zeta function have real part at least $-\frac{n+1}{2}$. Thus, for the single-polynomial case, Segers's theorem provides a better lower bound. Strengthening Theorem \ref{reverse_hensel} to imply Segers's theorem is a possible avenue for future research.
\end{remark}

\noindent Theorem \ref{reverse_hensel} provides us with sufficient structural conditions to show non-trivial lower bounds for an Igusa functions smallest real pole. Note that the theorem requires that the polynomial system $f$ be `highly singular', in the sense that the valuation of the Jacobian's elementary divisors of $f$ at a point $a$ should be bounded \textit{from below} with respect to $f(a)$. This stands in stark contrast with Hensel's Lemma, which assumes essentially the opposite and provides uniqueness to lifts of $p$-adically approximate solutions.

\begin{theorem}[Hensel's Lemma]
    Let $f\in \O[x_1, \dots,x_n]$ and $\overline{a} \in \O^n$ satisfy
    $$ 2\cdot \nu(\det J_f(\overline{a})) < \nu(f(\overline{a})). $$
    Then there exists a unique $\overline{b} \in \O^n$, such that $f(\overline{b}) = 0$ and $\nu(\overline{a} - \overline{b}) > \nu(\det J
    _f(\overline{a}))$.
\end{theorem}

\noindent In the Henselian framework, one requires $f$ to be sufficiently non-singular at a given point $\overline{a} \in \O^n$ to study $I_f(s)$. When this condition is not fulfilled, two common approaches are either to apply Hironaka's desingularization and study the Igusa zeta function of the resulting system instead \cite{igusa_intro, veys_poles_curves_90, meuser_poles_curves}, or to impose certain non-degeneracy conditions \cite{zuniga_non_deg_homog, zuniga_semiquasihomog, veys_zuniga_analytic, TR14, rossmann_framework}. However, for large and strongly singular systems, the former tends to be prohibitively difficult and the latter not applicable. That is why we took an approach in the style of Theorem \ref{reverse_hensel} in this paper.

Before explaining how the polynomials \ref{klm_poly} studied in Section \ref{graded_algebras_section} fit within the scope of Theorem \ref{reverse_hensel}, we consider another class of examples for which the theorem applies. The following is a consequence of Theorem \ref{reverse_hensel}.

\begin{example}\label{homog_deg2_corr}
    Let $f_1,\dots, f_m \in \O[x_1, \dots, x_n]$ be a system of homogeneous polynomials of degree at least $2$. Then all poles of $I_f(s)$ have real part at least $-n + \frac{n}{n+1}$.
\end{example}

\begin{proof}
    We use the same notation as in the proof of Theorem \ref{reverse_hensel}. Let $a \in \O^n$ be such that $f(a) \equiv 0\bmod \P^{2i}$. By homogeneity, $f(\lambda \cdot a) \equiv 0 \bmod \P^{2i}$ for any $\lambda \in \O^\times$. Observe that the valuation of the size of the orbit $\O^{\times} \cdot a$ is $2i-1-\nu(a)$. If $\nu(a) < \frac{2i}{n+1}$, this means that $v_q\left( \# (\O^{\times} \cdot a) \right) \geq 2i-\frac{2i}{n+1} = \frac{n}{n+1}2i$, as needed to be proven.

    On the other hand, if $\nu(a) \geq \frac{2i}{n+1}$, note that every entry of $J_f(a)$ has valuation at least $\nu(a)$. This is a consequence of the fact that every monomial in the polynomial system $f$ has degree at least $2$. Hence, every elementary divisor of $J_f(a)$ has valuation at least $\nu(a)$. Therefore, by the same reasoning as in the proof of Theorem \ref{reverse_hensel}, the number of $b \in \O^n$ such that $f(a + \pi^i b) \equiv 0 \bmod \P^{2i}$ has valuation at least $n\cdot \nu(a) \geq \frac{n}{n+1}2i $.
\end{proof}

\noindent The lower bound given in this example can undoubtedly be improved, but we limit ourselves to the present approach, because the proof is succinct and mirrors the structure of the proof of Theorem \ref{any_type_theorem} in Subsection \ref{any_type_subsection}. The proof of Example \ref{homog_deg2_corr} distinguishes two cases: $\nu(a)$ small, and $\nu(a)$ large. 

In the first case, we use the homogeneity of the system $f$ to apply a multiplicative action by $\O^\times$. This reflects the approach taken in \ref{mult_subsection}, where we used weighted homogeneity of the polynomial expressions \ref{klm_poly} to the same effect. Note that the valuation of the size of the orbit $\O^\times \cdot a$ is larger, the smaller $\nu(a)$. This observation corresponds to Lemmas \ref{val_beta_influenced} and \ref{val_beta-1_influenced}, where we proved that if certain entries of $\beta$ and $\beta^{-1}$ have low valuation, then the valuation of the orbit size of $\Lambda_\beta$ by the multiplicative action \ref{mult_action}, must be large.

In the other case, we apply an additive action by $\O^n$ given by $a \mapsto a + \pi^ib$. Here we use the assumption that the degree of every monomial in $f$ is at least $2$, whereby any linear term in $b$ has a coefficient with valuation at least $\nu(a)$. This corresponds to the action described in \ref{add_subsection}. Because we only assumed that the polynomial expressions \ref{klm_poly} are \textit{weighted} homogeneous and, in addition, did not place any restrictions on the weights or degrees, we had to be more precise in defining the additive action \ref{add_action}, than we were in the proof of Example \ref{homog_deg2_corr}. In particular, we had to narrow it to an action by $\O$ instead of $\O^n$. The idea, however, remains the same: the higher the valuations of certain entries of $\beta$ and $\beta^{-1}$, the more freely the additive action may change $\Lambda_\beta$, without changing its weight, as made explicit by Lemma \ref{any_type_weight_additive_theorem}, resulting in a higher valuation of the additive orbit size.

Both for Example \ref{homog_deg2_corr} as for Theorem \ref{any_type_theorem}, the proof relies on the interplay between the multiplicative and additive actions. In particular, if the valuation of the orbit size of the former is small, then that of the latter must be large. Therefore, either valuation must always be large. In general, the idea is that if any $p$-adically approximate solution $a$ to $f$ can be acted upon such that the first-order perturbation to $f(a)$ is sufficiently $p$-adically insignificant (read $J_f(a)b \equiv 0 \bmod p^i$), then a non-trivial lower bound to the associated Igusa zeta function $I_f(s)$ can be found.

\subsection*{Acknowledgements}

I thank my PhD-supervisor Christopher Voll, for pointing me to the questions answered and left unanswered in this paper, and for proof-reading several iterations. I am also grateful to Tobias Rossmann for discussing his many conjectures and ideas with me. Likewise, I thank Moritz Petschick and José Pedro Quintanilha for some interesting conversations and insights.

\printbibliography

\end{document}